\documentclass[final, 3p, times]{elsarticle}%
\usepackage{graphicx}
\usepackage{xcolor}
\usepackage{colortbl}
\usepackage{epstopdf}
\usepackage{amssymb}
\usepackage{amsthm}
\usepackage{amsmath}
\usepackage{color}
\usepackage{dsfont}
\usepackage{mathtools}
\usepackage{extarrows}
\usepackage{hyperref}
\usepackage{bm}
\usepackage{caption}
\usepackage{float}
\usepackage{verbatim}
\usepackage{epsfig}
\usepackage{multirow}
\usepackage{subfig}
\usepackage[section]{placeins}
\newcommand{\diff}{\,\mathrm{d}}

\numberwithin{equation}{section}%
\newtheorem{lem}{Lemma}[section]
\newtheorem{thm}{Theorem}[section]
\newtheorem{pro}{Proposition}[section]
\newtheorem{Def}{Definition}[section]
\newtheorem{Assump}{Assumption}[section]
\newtheorem{rmk}{Remark}[section]

\biboptions{compress}

\begin{document}
\begin{frontmatter}
\title{Convergence to equilibrium for fully discretizations of nonlocal Cahn-Hilliard equation}

\tnotetext[label1]{The research of Dongling Wang is supported in part by the National Natural Science Foundation of China under grants 12271463. The work of Danni Zhang is supported by Hunan Provincial Innovation Foundation For Postgraduate (No: CX20230625). \\ Declarations of interest: none.}

\author[XTU]{Danni Zhang} 
\ead{zhdanni0807@smail.xtu.edu.cn;}
\author[XTU]{Dongling Wang\corref{mycorrespondingauthor}}
\ead{wdymath@xtu.edu.cn; ORCID 0000-0001-8509-2837}
\cortext[mycorrespondingauthor]{Corresponding author. }

\address[XTU]{Hunan Key Laboratory for Computation and Simulation in Science and Engineering, School of Mathematics and Computational Science, Xiangtan University, Xiangtan, Hunan 411105, China}

\begin{abstract}
The study of long-term dynamics for numerical solutions of nonlinear evolution equations, particularly phase field models, has consistently garnered considerable attention. The Cahn-Hilliard (CH) equation is one of the most important phase field models and is widely applied in materials science. In order to more accurately describe the practical phenomena in material microstructural phase transitions, the Nonlocal Cahn-Hilliard (N-CH) equation incorporates a finite range of spatial nonlocal interactions is introduced, which is a generalization of the classic CH equation.

However, compared to its classic counterpart, it is very challenging to investigate the long-term asymptotic behavior of solution to the N-CH equation due to the complexity of the nonlocal integral term and the lack of high-order diffusion term.
In this paper, we consider first-order and second-order temporal discretization methods for the N-CH equation, respectively, while utilizing a second-order finite difference method for spatial approximation to construct the energy stable fully discrete numerical schemes. Based on energy stability and the {\L}ojasiewicz inequality, we rigorously prove that the numerical solutions of these fully discrete numerical schemes converge to equilibrium as time goes to infinity.
\end{abstract}

\end{frontmatter}

\section{Introduction}
In this work, we consider the following Nonlocal Cahn-Hilliard (N-CH) equation with the periodic boundary condition~\cite{Bates1,Du1}:
\begin{equation}\label{equ-1}
\begin{cases}
\partial_t u= \Delta \omega,\\
\omega=F'(u)+\varepsilon^2 (J\ast1)u-\varepsilon^2 (J\ast u),
\end{cases}\quad \mathbf{x}\in\Omega,\:t>0,
\end{equation}
where $\Omega=\prod_{i=1}^d(0,L_i)$ is a rectangular domain in $\mathbb{R}^{d}(d=1,2,3)$, $\varepsilon>0$ is the interfacial parameter, and the unknown functions are the order parameter $u(\mathbf{x},t)$ is subject to the initial condition $u(\mathbf{x},0)=u_0(\mathbf{x})$ and the chemical potential $\omega(\mathbf{x},t)$.
The nonlinearity $F$ is the free energy potential. A thermodynamically relevant potential $F$ is given by the  logarithmic function
 $F(r)=(1+r)\ln(1+r)+(1-r)\ln(1-r)-\frac{\theta}{2} r^2$ with $\theta\geq 2$ and $r\in(-1,1)$,
 which is often approximated by a regular double-well potential function,
 \begin{equation}\label{pot1}
 F(r)=\frac{1}{4}(r^2-1)^2,\quad r\in \mathbb{R}.
 \end{equation}
 The interaction kernel $J$ is integrable and even, i.e., $J(\mathbf{x})=J(-\mathbf{x})$ for any $\mathbf{x}\in \Omega$, and
 \[
 \textrm{J}\ast 1:=\int_{\Omega} \textrm{J}(\mathbf{\mathbf{x}}-\mathbf{y})\diff \mathbf{y},\quad
 (\textrm{J}\ast u)(\mathbf{x}):=\int_{\Omega} \textrm{J}(\mathbf{\mathbf{x}}-\mathbf{y})u(\mathbf{y})\diff \mathbf{y}.
 \]
 For instance, we can take a Gaussian kernel $\textrm{J}(\mathbf{x})=c_Je^{-\xi|\mathbf{x}|^2},\xi>0$, or a Newtonian kernel $\textrm{J}(\mathbf{x})=c_J|\mathbf{x}|^{-1}$, if $d=3$, $\textrm{J}(\mathbf{x})=-c_J \log(|\mathbf{x}|)$, if $d=2$.

In this paper, we further assume the kernel $\textrm{J}$ in \eqref{equ-1} is sufficiently smooth and $\Omega$-periodic and that the convolution $\textrm{J}\ast u$ is periodic \cite{Guan1}, i.e.,
\[
(\textrm{J}\ast u)(\mathbf{x}):=\int_{\Omega} \textrm{J}(\mathbf{y})u(\mathbf{x}-\mathbf{y})\diff \mathbf{y}=\int_{\Omega} \textrm{J}(\mathbf{\mathbf{x}}-\mathbf{y})u(\mathbf{y})\diff \mathbf{y},\quad u\in L_{per}^2(\Omega),
\]
where $L_{per}^2(\Omega)$ denote the space of all periodic functions in $L^2(\Omega)$.
Clearly $\textrm{J}\ast 1=\int_{\Omega} \textrm{J}(\mathbf{x})\diff \mathbf{x}>0$ is a positive constant.
The N-CH equation~\eqref{equ-1}~can be viewed as the $H^{-1}$ gradient flow with respect to the following free energy functional
\begin{equation}\label{eq1}
    E(u)=\int_{\Omega} F(u) \diff \mathbf{x}+\frac{\varepsilon^2 (\textrm{J}\ast 1)}{2}\|u\|_{L^2}^2-\frac{\varepsilon^2 }{2}(u,\textrm{J}\ast u)_{L^2},
\end{equation}
where $\|\cdot\|_{L^2}$ and $(\cdot,\cdot)_{L^2}$ stand for the spatial $L^2(\Omega)$ norm and $L^2(\Omega)$ inner product, respectively.
A solution to N-CH equation \eqref{equ-1} satisfies the mass conservation, 
\begin{equation}\label{eq33}
   \int_{\Omega}u(\mathbf{x},t)\diff \mathbf{x}=\int_{\Omega}u(\mathbf{x},0)\diff \mathbf{x},\quad t\geq 0,
\end{equation}
and energy stability
\begin{equation}\label{eq333}
\frac{\diff E}{\diff t}=\left (\frac{\delta E}{\delta u}, \partial_t u\right)=- \|\nabla \omega\|^2_{L^2}\leq 0,\quad t>0.
\end{equation}

The N-CH equation \eqref{equ-1} can be written as
\begin{equation}
    \partial_tu=\nabla \cdot (a(u)\nabla u)-(\Delta \textrm{J})\ast u,
\end{equation}
where $a(u)=3u^2-1+\varepsilon^2 (\textrm{J}\ast 1)$ is referred as the diffusive mobility. In order to ensure \eqref{equ-1} positive diffusive and the solution becomes regular in time, we make the following assumption \cite{Bates2,Bates3}
\begin{equation}\label{eq3}
    \gamma:=\varepsilon^2(\textrm{J}\ast 1)-1>0.
\end{equation}
Otherwise, the solution may exhibit some regular behavior.
In this paper, we always assume that the kernel $\textrm{J}$ satifies the condition \eqref{eq3} and we choose \eqref{pot1} as the potential function in N-CH equation \eqref{equ-1}.
Note that the kernel $\textrm{J}$ is even, we can rewrite the energy~\eqref{eq1}~as
\begin{equation}\label{eq331}
    E(u)=\int_\Omega \left(F(u)+\frac{\varepsilon^2}{4}\int_\Omega \textrm{J}(\mathbf{x}-\mathbf{y})(u(\mathbf{x})-u(\mathbf{y}))^2\diff y\right)\diff \mathbf{x}.
\end{equation}
The second term in \eqref{eq331} typically represents the interaction energy density, which describes the long-range interactions between atoms at different positions. The kernel $\textrm{J}$ is used to measure the strength of the interaction.

If we further assume that $\textrm{J}$ has a finite second moment in $\Omega$, i.e., $\frac{1}{2}\int_{\Omega}J(\mathbf{x})|\mathbf{x}|^2\diff \mathbf{x}=1$.
Then, the classic Cahn-Hilliard (CH) equation \cite{CH}
\begin{equation}\label{eq-199}
\begin{cases}
\partial_t u= \Delta \omega,\\
\omega=F'(u)-\varepsilon^2 \Delta u,
\end{cases}\quad \mathbf{x}\in\Omega,\:t>0,
\end{equation}
is recovered from \eqref{equ-1} in the local limit \cite{Du1,Li1}.
The corresponding local energy functional is
\begin{equation}
    E_{local}(u)=\int_\Omega \left(F(u)+\frac{\varepsilon^2}{2}|\nabla u(x)|^2\right)\diff \mathbf{x}.
\end{equation}
The classic CH equation is an approximate form of the NCH equation under the assumption of short-range interactions.
The solution of CH equation \eqref{eq-199} is also mass conservative and energy stable.

It is well-known that the CH equation is one of the typical phase field models describing the kinetics of the two-phase separation.
The gradient flow structure of the CH equation ensures that its $\omega$-limit set only contains equilibrium points which are critical points of energy functional $E_{local}$ with prescribed mass. Rybka and Hoffmann \cite{RH} provided a rigorous proof that the solution to the CH equation with an analytic nonlinear term converges to an equilibrium as time goes to infinity, i.e. the $\omega$-limit set
\[
\omega(u)=\left\{ u_\ast:\exists t_n\to \infty \text{ such that } u(t_n)\to u_\ast\; \text{ in }\; C^{4,\, \mu}(\Omega) \right\}
\]
is a singleton, where $\mu\in (0,1)$ and $u_\ast$ is a solution of the stationary problem:
\begin{equation}\label{sta-1}
\Delta F'(u_\ast)-\varepsilon^2 \Delta^2 u_\ast=0 \text{ in } \;C^{0,\, \mu}(\Omega).
\end{equation}
Their proof is mainly based on energy stability and a {\L}ojasiewicz-Simon inequality \cite{Simon}, which is a generalization of the celebrated {\L}ojasiewicz inequality \cite{L1, Haraux} in infinite-dimensional space. The {\L}ojasiewicz-Simon inequality allow us to prove convergence to equilibrium without any knowledge on the set of equilibrium state sets.

There has been a significant amount of work on the convergence to equilibrium of numerical solutions for the classic CH equation. Merlet and Pierre \cite{Pierre2} first proved the sequence generated by the backward Euler scheme to CH equation with analytical nonlinearity converges to an equilibrium by means of the {\L}ojasiewicz inequality, with spatial discretization using the finite element method. Thanks to the spatial discretization, the {\L}ojasiewicz-Simon inequality reduces to the standard {\L}ojasiewicz inequality. For schemes different from the backward Euler method, the study on the convergence to equilibrium can be referred to \cite{Pierre3,Pierre1}. The proof also main relies on energy stability and the {\L}ojasiewicz type inequality.

In recent years, as an extension of the classic CH equation, the N-CH equation has rapidly developed in various research fields, such as materials science \cite{Bates1}, image segmentation \cite{Gajewski}, electrode materials \cite{Kamlah}, and tumor growth \cite{Fritz}. There have been many works on both theoretical analysis and numerical aspects for the N-CH equation. In the theoretical level, Bates and Han \cite{Bates2,Bates3} studied the well-posedness of the N-CH equation equipped with Neumann or Dirichlet boundary condition by assuming the kernel is integrable. They also showed the existence of absorbing sets in the $H^1$ norm. Guan et al. \cite{Guan1} claimed that the existence and uniqueness of the periodic solution to the N-CH equation may be established by using a similar technique. Agosti et al. \cite{Agosti} investigate a N-CH equation with a singular single-well potential and degenerate mobility. They prove the existence and uniqueness of weak solutions for spatial dimensions $d\leq3$.
Regarding the long-time behavior of the weak solution to the N-CH equation, Gal and Grasselli \cite{Gal} established the convergence to equilibrium as time tends to infinity, i.e. the $\omega$-limit set of the weak solution $u$ of N-CH equation
\[
\omega(u)=\left\{ u_\ast:\exists t_n\to \infty\; \text{such that} \;u(t_n)\to u_\ast\; \text{in}\; L^2(\Omega) \right\}
\]
is a singleton, where $u_\ast$ is a solution of the stationary problem:
\begin{equation}\label{sta-2}
\begin{cases}
F'(u_\ast)+\varepsilon^2 (\textrm{J}\ast1)u_\ast-\varepsilon^2 (\textrm{J}\ast u_\ast)=\omega_\ast \;\text{ a.e. }\;\text{in} \;\Omega,\\
\omega_\ast=\text{constant},\; \int_\Omega u_\ast \diff \mathbf{x}=\int_\Omega u_0 \diff \mathbf{x}.
\end{cases}
\end{equation}
The proof based on energy stability and a generalized {\L}ojasiewicz-Simon inequality \cite{Gajewski1,Londen}.

On the numerical analysis, numerous numerical discretization methods been developed that can inherit the energy stability of the N-CH equation. Such as convex splitting schemes \cite{Guan1,Guan2,Guan3,Sheng}, stabilized linear schemes \cite{Li1,Du1,li2023,li2024}, exponential time differencing schemes \cite{Zhang1, Zhang2}, the BDF methods \cite{BG, Zhao}, the operator splitting spectral schemes \cite{Zhai1, Zhai2} and so on.

It is natural to ask if the sequences generated by the fully discrete numerical schemes for the N-CH equation converge to equilibrium as time approaches infinity. However, the analysis of long-time behavior of numerical solutions to the N-CH equation faces significant challenges, owing to the complexity of the nonlocal integral term and the lack of high-order diffusion term. To date, there has no rigorous numerical analysis results.
Note that when the kernel $\textrm{J}$ satisfies the aforementioned conditions, the CH equation can be viewed as a local approximation of the N-CH equation. This observation naturally motivates extending the convergence to equilibrium analysis techniques developed for the discretized classic CH equation to the N-CH equation.
The main contributions of this paper can be summarized as follows:
\begin{itemize}
\item\textbf{First-order schemes:} We present a comprehensive analysis of convergence to equilibrium for three energy-stable first-order temporal discretization schemes, including the backward Euler scheme, the convex splitting scheme and stabilized linear semi-implicit (SSI1) scheme, applied to the N-CH equation with an analytic nonlinearity. A second order finite difference method is employed for spatial discretization. For the backward Euler scheme, we first establish the existence and uniqueness of the solution sequence. For all three schemes, we prove discrete mass conservation and energy stability. By combining the energy stability and the {\L}ojasiewicz inequality, we rigorously prove that the sequences generated by these fully discrete schemes converge to equilibrium as time tends to infinity.
\item\textbf{Second-order schemes:} We develop and analyze two second-order temporal discretization schemes, including the BDF2 scheme and the linearly implicit (2LI) scheme, both of which are proved to be mass conservative and modified energy stable. Furthermore, we establish the unique solvability of the BDF2 scheme. Again leveraging energy stability and the {\L}ojasiewicz inequality, we prove the convergence of the numerical solution sequences to a steady state as $n \to \infty$.
\end{itemize}

The paper is organized as follows.
In Section \ref{sec2}, we introduce the discretizations of the spatial differential operators and the periodic convolution for the N-CH equation.
In Section \ref{sec3}, we give three first-order fully discrete numerical schemes for the N-CH equation and establish the unique solvability, mass conservation and energy stability. We rigorously prove the sequences generated by the proposed three numerical schemes converge to the equilibrium as time goes to infinity.  Two second-order schemes is presented in Section \ref{sec4}.
Finally, some concluding remarks are given in Section \ref{sec5}.

\section{Spatial discretization of the finite difference method}\label{sec2}
In this section, we employ the finite difference approximation on a staggered grid
developed in \cite{Wang1,Wang2,Guan1}
 for the spatial differential operators and the periodic convolution of the N-CH equation. Without loss of generality, we only discuss the two-dimensional scenario, the one-dimensional and three-dimensional cases are similar.

\subsection{Finite difference approximations of the spatial differential operators}
Our primary goal in this subsection is to define the grid functions, the difference discretization of the spatial differential operators as well as summation-by-parts formulae in 2D space. Suppose $\Omega=(0,L)\times (0,L)$. Let $h=\frac{L}{N}$, where $N\in\mathbb{Z}^+$ and $h>0$ is the spatial step size. We define the following three sets:
\[
V_N=\{i\cdot h|i=0,1,\cdots,N\},\;
C_N=\left\{\left(i-\frac{1}{2}\right)\cdot h|i=1,2,\cdots,N\right\},\;
C_{\bar{N}}=\left\{\left(i-\frac{1}{2}\right)\cdot h|i=0,1,\cdots,N+1\right\}.
\]
The elements in $V_N$ are edge-centered points, while the elements in $C_N$ and $C_{\bar{N}}$ are cell-centered points. Define the 2D grid function spaces:
\begin{align}
    & \mathcal{C}_{N\times N}=\{\phi: C_N \times  C_N\to \mathbb{R}\},\;\mathcal{C}_{\bar{N}\times \bar{N}}=\{ \phi: C_{\bar{N}} \times  C_{\bar{N}}\to \mathbb{R}\},\; V_{N\times N}=\{f:V_N\times V_N\to \mathbb{R}\},\\
    &\mathcal{E}^{x}_{N\times N}=\{f:V_N\times  C_N\to \mathbb{R}\},\; \mathcal{E}^{y}_{N\times N}=\{f:C_N\times  V_N\to \mathbb{R}\},\; \vec{\mathcal{E}}_{N\times N}=\mathcal{E}^{x}_{N\times N}\times \mathcal{E}^{y}_{N\times N}.
\end{align}
The functions of $\mathcal{C}_{N\times N}$ and $\mathcal{C}_{\bar{N}\times \bar{N}}$ are called cell-centered functions. In component form these functions are defined via $\phi_{i,j}:=\phi(x_i,y_j)$, where $x_i=(i-1/2)\cdot h, y_j=(j-1/2)\cdot h$.
The functions of $\mathcal{E}^{x}_{N\times N}$ and
$\mathcal{E}^{y}_{N\times N}$ are called edge-centered functions. In component form these functions are identified via $f_{i+1/2,j}:=f(x_{i+1/2},y_j),f_{i,j+1/2}:=f(x_i,y_{j+1/2})$, respectively. The functions of $V_{N\times N}$ are called vertex-centered functions are defined componentwise by $f_{i+1/2,j+1/2}:=f(x_{i+1/2},y_{j+1/2})$.

A cell-centered function $\phi\in \mathcal{C}_{\bar{N}\times \bar{N}}$ is said to satisfy periodic boundary conditions if and only if,
\begin{align}
    \phi_{N,j}=\phi_{0,j},\quad \phi_{N+1,j}=\phi_{1,j},\quad
    \phi_{i,N}=\phi_{i,0}, \quad \phi_{i,N+1}=\phi_{i,1}.
\end{align}
And we denote the set of such functions as $\mathcal{C}_{\bar{N}\times \bar{N}}^{per}$. The edge-centered functions and vertex-centered functions that satisfy periodic boundary conditions are similarly defined. The domain of the cell-centered functions, edge-centered functions and vertex-centered functions subject to the periodic boundary conditions are shown in Figure \ref{figure1}.
\begin{figure}[!htb]
	\centering
    \vspace{-0.35cm}
    \setlength{\abovecaptionskip}{-1.7cm}
     {\includegraphics[width=0.6\textwidth]{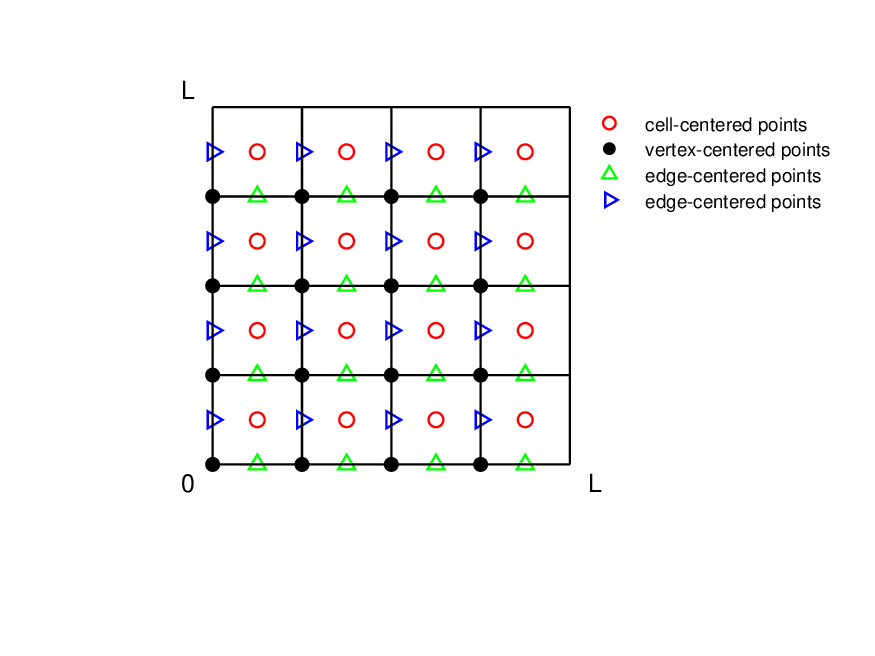}}
     \caption{The domain of various grid functions when $N=4$.}
     \label{figure1}
\end{figure}

The edge-to-center difference operators $\diff_x:\mathcal{E}^{x}_{N\times N}\to \mathcal{C}_{N\times N}$ and $\diff_y:\mathcal{E}^{y}_{N\times N}\to \mathcal{C}_{N\times N}$; and the center-to-edge difference operators $D_x:\mathcal{C}_{\bar{N}\times \bar{N}}\to \mathcal{E}^{x}_{N\times N}$ and $D_y:\mathcal{C}_{\bar{N}\times \bar{N}}\to \mathcal{E}^{y}_{N\times N}$, are defined componentwise via
\begin{align*}
    &\diff_xf_{i,j}=\frac{1}{h}(f_{i+1/2,j}-f_{i-1/2,j}),\quad \diff_yf_{i,j}=\frac{1}{h}(f_{i,j+1/2}-f_{i,j-1/2}),\\
    &D_x \phi_{i+1/2,j}=\frac{1}{h}(\phi_{i+1,j}-\phi_{i,j}),\quad D_y \phi_{i,j+1/2}=\frac{1}{h}(\phi_{i,j+1}-\phi_{i,j}).
\end{align*}
For any $\boldsymbol{f}=(f^x,f^y)^T$ and $\phi \in \mathcal{C}_{\bar{N}\times \bar{N}}$, the discrete divergence operator $\nabla_h\cdot:\vec{\mathcal{E}}_{N\times N}\to \mathcal{C}_{N\times N}$, gradient operator $\nabla_h:\mathcal{C}_{\bar{N}\times \bar{N}}\to \vec{\mathcal{E}}_{N\times N}$ and Laplace operator $\Delta_h:\mathcal{C}_{\bar{N}\times \bar{N}} \to \mathcal{C}_{N\times N}$ are given respectively by
\[
\nabla_h\cdot \boldsymbol{f}:=\diff_xf^x+\diff_yf^y,\quad \nabla_h\phi:=(D_x\phi,D_y\phi)^T,\quad \Delta_h\phi=\nabla_h\cdot(\nabla_h\phi).
\]
We define the following inner-products:
\begin{align}
    &(\phi\|\psi)= \sum_{i=1}^N\sum_{j=1}^N\phi_{i,j}\psi_{i,j}   ,\quad \phi,\psi \in \mathcal{C}_{\bar{N}\times \bar{N}},\\
    &[f\|g]_{x}=\frac{1}{2}\sum_{i=1}^N\sum_{j=1}^N(f_{i+1/2,j}g_{i+1/2,j}+f_{i-1/2,j}g_{i-1/2,j}),\quad f,g\in \mathcal{E}^{x}_{N\times N},\\
    &[f\|g]_{y}=\frac{1}{2}\sum_{i=1}^N\sum_{j=1}^N(f_{i,j+1/2}g_{i,j+1/2}+f_{i,j-1/2}g_{i,j-1/2}),\quad f,g\in \mathcal{E}^{y}_{N\times N},\\
   &(\nabla_h \phi\|\nabla_h\psi)=[D_x \phi\|D_x\psi]_{x}+[D_y \phi\|D_y\psi]_{y},\quad \phi,\psi \in \mathcal{C}_{\bar{N}\times \bar{N}}.
\end{align}
Using the above definitions, we can obtain the following summation-by-parts formulae \cite{Wang1,Wang2}.
\begin{pro}
If $\phi,\psi\in \mathcal{C}_{\bar{N}\times \bar{N}}^{per}$, then
$
    h^2(\nabla_h \phi\|\nabla_h\psi)=-h^2(\phi\|\Delta_h\psi),\quad
    h^2(\phi\|\Delta_h\psi)=h^2(\Delta_h\phi\|\psi).
    $
\end{pro}

For $\phi\in \mathcal{C}_{\bar{N}\times \bar{N}}$, we introduce the following norms
\begin{align*}
\|\phi\|_2:=\sqrt{h^2(\phi\|\phi)},\quad \|\phi\|_4:={}^4\sqrt{h^2(\phi^4\| \mathbf{1})},\quad \|\nabla_h\phi\|_2:=\sqrt{h^2(\nabla_h\phi\|\nabla_h\phi)}.
\end{align*}
Noting that the N-CH equation is mass conservative, without loss of generality, we assume that the mean of $u$ is zero and  consider the zero-mean periodic grid functions space:
\[
\mathcal{M}_h^0:=\{u\in \mathcal{C}_{N\times N}^{per}| \;(u\| \mathbf{1})=0\},
\]
which implies that $(-\Delta_h)|_{\mathcal{M}_h^0}$ is self-adjoint and positive definite. Thus, the inversion of $(-\Delta_h)^{-1}$ exists on $\mathcal{M}_h^0$. We can define the discrete norm $\|\cdot\|_{-1}$,
\[
\|u\|_{-1}:=\sqrt{h^2((-\Delta_h)^{-1}u\|u)},\quad \forall u\in \mathcal{M}_h^0.
\]

Denote $A_h\in \mathbb{R}^{N^2\times N^2}$ is the discrete matric of $(-\Delta_h)$. We define the matrix $D_h$ of order $N$:
\[
D_h=-\frac{1}{h^2}
\begin{pmatrix}
-2 & 1 & & & 1 \\
1 & -2 & 1 & &  \\
 & \ddots & \ddots & \ddots & \\
 & & 1 & -2 & 1 \\
1 & & & 1 & -2
\end{pmatrix}.
\]
By arranging the cell-centered points in lexicographical order, we deduce
\begin{equation}\label{eq29}
A_h=I\otimes D_h+D_h \otimes I,
\end{equation}
where $I$ is the identity matrix of order $N$. It is easy to see $A_h$ is a symmetric diagonally dominant matrix, and its eigenvalues are
\[
\lambda_{k,l}^{A_h}=\frac{2}{h^2}\left(2-\cos\left(\frac{2k\pi}{N}\right)-\cos\left(\frac{2l\pi}{N}\right)\right)\geq 0,\quad 1\leq k,l\leq N.
\]
Note that $A_h$ is positive semi-definite matrix and $0$ is a simple eigenvalue of $A_h$ with the eigenvector $E_1=(1,1,\cdots,1)^T$.

\subsection{Numerical approximation of the periodic convolution and key inequalities}
Here we present the definition of the discrete periodic convolution on a 2D periodic grid, and introduce some useful inequalities. Suppose $\phi\in \mathcal{C}_{N\times N}$ and $f\in V_{N\times N}$ are periodic. Thus, the discrete convolution operator $[f\circledast \phi]: V_{N\times N}\times \mathcal{C}_{N\times N}\to \mathcal{C}_{N\times N}$ is defined componentwise by \cite{Guan1}:
\begin{equation}\label{eq23}
    [f\circledast \phi]_{i,j}=h^2\sum_{k=1}^N\sum_{l=1}^Nf_{k+\frac{1}{2},l+\frac{1}{2}}\phi_{i-k,j-l}.
\end{equation}

Denote the $A_J\in \mathbb{R}^{N^2\times N^2}$ is a nonlocal matrix associated of the discrete convolution term. It follows from \eqref{eq23} that
\[
[J\circledast\mathbf{1}]\phi_{i,j}-[J\circledast \phi]_{i,j}=h^2\sum_{k=1}^N\sum_{l=1}^NJ_{k+\frac{1}{2},l+\frac{1}{2}}\phi_{i,j}-h^2\sum_{k=1}^N\sum_{l=1}^NJ_{k+\frac{1}{2},l+\frac{1}{2}}\phi_{i-k,j-l}.
\]
Then, we can explicitly represent the nonlocal matrix $A_J$ as the block circulant matrix:
\begin{equation}\label{eq28}
A_J=-h^{2}\cdot\operatorname{circ}\left(E_{N+\frac{1}{2}},E_{N-1+\frac{1}{2}},E_{N-2+\frac{1}{2}}\cdots,E_{1+\frac{1}{2}}\right),
\end{equation}
where
\[
E_{q}=
\begin{cases}
\operatorname{circ}\left(-c_0,J_{q,N-1+\frac{1}{2}},J_{q,N-2+\frac{1}{2}},\cdots,J_{q,1+\frac{1}{2}}\right), & q=N+\frac{1}{2}, \\
\operatorname{circ}\left(J_{q,N+\frac{1}{2}},J_{q,N-1+\frac{1}{2}},J_{q,N-2+\frac{1}{2}},\cdots,J_{q,1+\frac{1}{2}}\right), & q\neq N+\frac{1}{2},
\end{cases}
\]
where $c_0=\sum_{k=1}^N\sum_{l=1}^NJ_{k+\frac{1}{2},l+\frac{1}{2}}-J_{N+\frac{1}{2},N+\frac{1}{2}}$.
It is easy to verify that the matrix $A_J$ is symmetric and diagonally dominant and its eigenvalues can drive as:
\[
\lambda_{k,l}^{A_J}=h^2\sum_{i=1}^N\sum_{j=1}^NJ_{i+\frac{1}{2},j+\frac{1}{2}}\left(1-\cos\left(\frac{2\pi ki}{N}+\frac{2\pi lj}{N}\right)\right),\quad 1\leq k,l\leq N.
\]
We can see that the $A_J$ is a positive semi-definite matrix and has one eigenvalue is $0$. Noting that the sum of each row element of $A_J$ is $0$, the eigenvector corresponding to the eigenvalue of $0$ is $E_1$.
Since the continuous kernel function $\textrm{J}$ satisfies the assumption \eqref{eq3}, the consistency of the discrete convolution implies
\begin{equation}\label{eq-assum1}
    \gamma_0:=\varepsilon^2[J\circledast \mathbf{1}]-1>0.
\end{equation}

Next, we present two lemmas on some fundamental properties of discrete periodic convolution operators; These will be very helpful for our subsequent analysis.
\begin{lem}\label{lem1}\cite{Guan1}
Suppose $\phi,\psi\in \mathcal{C}_{\bar{N}\times \bar{N}}^{per}$ and $f\in V_{N\times N}$ is periodic and even, i.e., $f_{i+\frac{1}{2},j+\frac{1}{2}}=f_{-i+\frac{1}{2},-j+\frac{1}{2}}$, for all $i,j\in \mathbb{Z}$, then
\begin{equation}\label{eq-5}
        (\phi\| [f\circledast \psi])=(\psi\| [f\circledast \phi]).
\end{equation}
In addition, if $f\geq 0$ then
\begin{equation}\label{eq-6}
        |(\phi\| [f\circledast \psi])\leq [f\circledast \textbf{1}]\left(\frac{\alpha}{2}(\phi\|\phi)+\frac{1}{2\alpha}(\psi\|\psi)\right),\quad \forall \alpha >0.
    \end{equation}
\end{lem}
\begin{lem}\label{lem2}\cite{Guan1}
If $\phi,\psi\in \mathcal{C}_{\bar{N}\times \bar{N}}^{per}$. Assume that $\textrm{f}\in C^{\infty}(\Omega)$ is periodic and even, and define its grid restriction via $f_{i+\frac{1}{2},j+\frac{1}{2}}:=\textrm{f}(x_{i+\frac{1}{2}},y_{j+\frac{1}{2}})$, so that $f\in V_{N\times N}$. We have
\begin{equation}
        -2h^2([f\circledast \psi]\| \Delta_h \phi)\leq \frac{C}{\alpha}\|\psi\|_2^2+\alpha\|\nabla_h \phi\|^2_2,\quad \forall \alpha>0,
    \end{equation}
    where $C>0$ depends on $\textrm{f}$ but is independent of $h$.
\end{lem}

\section{Convergence to equilibrium for the first-order fully discrete schemes of the N-CH equation}\label{sec3}
In this section, we mainly consider three energy-stable first-order fully discrete schemes (the backward Euler scheme, convex splitting scheme, stabilized linear semi-implicit (SSI1) scheme)
for the N-CH equation and prove these schemes generate sequences converge to equilibrium as time tends to infinity.
For the first-order backward Euler scheme, we prove the uniquely solvability, mass conservation and energy stability. Subsequently, by leveraging the energy stability and the {\L}ojasiewicz inequality, we give detailed proof for the sequence generates by this scheme converges to a steady state as $n\to \infty$.
For the first-order convex splitting scheme and the SSI1 scheme, we also prove that the corresponding solution sequences converge to equilibrium.

Our proof primarily relies on energy stability and the following {\L}ojasiewicz inequality \eqref{eq8}.

\begin{Def}
We say $G\in C^1(\mathbb{R}^M,\mathbb{R})$ satisfies the {\L}ojasiewicz inequality at $U^{\ast}\in \mathbb{R}^M$ if there exists $\nu \in (0,1/2], \gamma_1>0$ and $\sigma>0$ such that for all $V\in \mathbb{R}^M$,
\begin{equation}\label{eq8}
     \|V-U^{\ast}\|<\sigma\Rightarrow |G(V)-G(U^{\ast})|^{1-\nu}\leq \gamma_1\|\nabla G(V)\|,
\end{equation}
where $\|\cdot\|$ denotes the Euclidean norm in $\mathbb{R}^{M}$ and $\nabla G(U)=(\partial_1 G(U),\cdots,\partial_{M} G(U))^T,$.
\end{Def}

\begin{thm}[{\L}ojasiewicz Theorem \cite{L1}]\label{lem3}
If $G\in C^1(\mathbb{R}^M,\mathbb{R})$ is real analytic, thus $G$ satisfies the {\L}ojasiewicz inequality at each point $U^{\ast}\in \mathbb{R}^M$.
\end{thm}

The following theorem is the key to proving convergence to equilibrium.
\begin{thm}\label{the1}\cite{Pierre1}
If $G\in C^1(\mathbb{R}^M,\mathbb{R})$ is real analytic. Consider a bounded sequence $(U^n)_{n\geq 0}$ in $\mathbb{R}^M$ which satifies the following conditions:
\begin{enumerate}
    \item[\textbf{H1:}] There exists a constant $c_2>0$ such that for every $n \in \mathbb{N}$,
    \[
    G(U^n)-G(U^{n+1})\geq c_2\|U^{n+1}-U^n\|^2;
    \]
    \item[\textbf{H2:}] There exists a constant $c_3>0$ such that for every $n \in \mathbb{N}$,
    \[
    \|\nabla G(U^{n+1})\|\leq c_3\|U^{n+1}-U^n\|.
    \]
\end{enumerate}
Thus, the whole sequence $(U^n)_{n\geq 0}$ converges in $\mathbb{R}^M$.
\end{thm}
\subsection{The first-order backward Euler scheme}
In this subsection, we focus on the fully discrete first-order backward Euler scheme for the N-CH equation. We prove the unique solvability, mass conservation and energy stability of the proposed numerical scheme. Futhermore, we utilize the energy stability, the {\L}ojasiewicz inequality \eqref{eq8} and Theorem \ref{the1} to demonstrate that the sequence generated by this scheme converges to an equilibrium as $n\to \infty$.

The fully discrete first-order backward Euler scheme for the N-CH equation is constructed as follows: let $u_h^0=(u_{ij}^0)\in \mathcal{C}_{N\times N}^{per}$, and for $n\geq 0$, let $u_h^{n+1}=(u_{ij}^{n+1})\in \mathcal{C}_{N\times N}^{per},\; \omega_h^{n+1}=(\omega_{ij}^{n+1}) \in \mathcal{C}_{N\times N}^{per}$ solve
\begin{eqnarray}\label{eq-2}
\left\{
\begin{array}{lllll}
\frac{u_h^{n+1}-u_h^n}{\tau}= \Delta_h \omega_h^{n+1}, \hspace{0.7cm}\ \\
\omega_h^{n+1} =F'(u_h^{n+1})+\varepsilon^2 [J\circledast \mathbf{1}]u_h^{n+1}-\varepsilon^2[J\circledast u_h^{n+1}].
\end{array}
\right.
\end{eqnarray}
The corresponding discrete energy is defined as
\begin{equation}\label{1}
    E_h(u_h)=h^2(F(u_h)\| \mathbf{1})+\frac{\varepsilon^2 [J\circledast \mathbf{1}]}{2}\| u_h\|_2^2-\frac{\varepsilon^2}{2}h^2(u_h\| [J\circledast u_h]), \quad \forall u_h\in \mathcal{C}_{N\times N}^{per}.
\end{equation}

The following theorem illustrates that the first-order backward Euler scheme \eqref{eq-2} is mass conservative, uniquely solvable and energy stable.
\begin{thm}\label{the4}
The total mass for the scheme \eqref{eq-2} is conserved, i.e., $ (u_h^{n+1}-u_h^n\|\textbf{1} )=0$ for $n\geq 0$.
 Moreover, if the time step size $\tau$ satifies
\begin{equation}\label{tau1}
\tau\leq \frac{2\gamma_0}{C_\textrm{J}\varepsilon^4},
\end{equation}
where $C_\textrm{J}>0$ is a constant that depends on kernel function $\textrm{J}$ but is independent of $h$. Thus, the first-order backward Euler scheme \eqref{eq-2} is uniquely solvable and energy stable, i.e.,
\[
E_h(u_h^{n+1})\leq E_h(u_h^n),\quad \forall n\geq 0.
\]
\end{thm}
\begin{proof}
Taking the discrete inner product of the first equation in \eqref{eq-2} with $\mathbf{1}$, we obtain $(u_h^{n+1} -u_h^n\|\mathbf{1})=0$, which implies the scheme \eqref{eq-2} is mass conservative.
Next, We prove that the scheme \eqref{eq-2} is uniquely solvable and energy stable.

We define the energy functional $G[z]$ on the space $\mathcal{M}_h^{\ast}=\{z\in \mathcal{C}_{N\times N}^{per}| \; (z\|\mathbf{1} )= ( u_h^n \| \mathbf{1})\}$,
\[
    G[z]=\frac{1}{2\tau}\|z-u_h^n\|_{-1}^2+E_h(u_h).
\]
According to Lemma \ref{lem1}, we have
\begin{align*}
    G[z] &=\frac{1}{2\tau}\|z-u_h^n\|_{-1}^2+\frac{\varepsilon^2[J\circledast \mathbf{1}]}{2}\|z\|_2^2-\frac{\varepsilon^2 h^2}{2}(z\| [J\circledast z])+\frac{1}{4}\|z\|_4^4-\frac{1}{2}\|z\|_2^2+\frac{1}{4}|\Omega|\\
    &\geq \frac{1}{2\tau}\|z-u_h^n\|_{-1}^2+\frac{1}{4}\|z\|_4^4-\frac{1}{2}\|z\|_2^2+\frac{1}{4}|\Omega|\\
    &\geq \frac{1}{4}(4\|z\|_2^2-4|\Omega|)-\frac{1}{2}\|z\|_2^2+\frac{1}{4}|\Omega|\\
    &=\frac{1}{2}\|z\|_2^2-\frac{3}{4} |\Omega|.
\end{align*}
Clearly, the functional $G[z]$ is coercive; hence, a minimum of $G[z]$ exists. We now prove that $G[z]$ is strictly convex for any $s\in \mathbb{R}$ and $\psi \in \mathcal{M}_h^0$, which will imply the uniqueness of the minimizer. Simple calculations yield
\begin{equation}\label{eq5}
\frac{\diff^2 G}{\diff s^2}[z+s\psi]\big|_{s=0}=\frac{1}{\tau}\|\psi\|_{-1}^2+\varepsilon^2 [J\circledast \mathbf{1}]\|\psi\|^2_2-\varepsilon^2 h^2([J\circledast \psi]\|\psi)+3\|z\psi\|_2^2-\|\psi\|_2^2.
\end{equation}
An application of Lemma \ref{lem2} to the term $([J\circledast \psi]\|\psi)$ in \eqref{eq5} shows
\begin{align}\label{eq4}
h^2([J\circledast \psi]\|\psi)&=-h^2([J\circledast \psi]\|\Delta_h((-\Delta_h)^{-1}\psi))\nonumber\\
&\leq \frac{C_\textrm{J}\tau\varepsilon^2}{2} \|\psi\|_2^2+\frac{1}{2\tau\varepsilon^2}\|\nabla_h(-\Delta_h)^{-1}\psi\|_2^2\nonumber\\
&\leq \frac{C_\textrm{J}\tau\varepsilon^2}{2} \|\psi\|_2^2+\frac{1}{2\tau\varepsilon^2}\|\psi\|_{-1}^2.
\end{align}
Combining \eqref{eq5}-\eqref{eq4} and using the condition that $\tau$ satisfies \eqref{tau1}, we obtain,
\begin{align*}
\frac{\diff^2 G}{\diff s^2}[z+s\psi]\big|_{s=0}&\geq \frac{1}{2\tau}\|\psi\|_{-1}^2+\left(\varepsilon^2 [J\circledast \mathbf{1}]-\frac{C_\textrm{J}\tau\varepsilon^4}{2}-1\right)\|\psi\|_2^2+3\|z\psi\|_2^2\\
&=\frac{1}{2\tau}\|\psi\|_{-1}^2+\left(\gamma_0-\frac{C_\textrm{J}\tau\varepsilon^4}{2}\right)\|\psi\|_2^2+3\|z\psi\|_2^2>0,
\end{align*}
where we have used the assumption \eqref{eq-assum1}. Therefore, the functional $G[z]$ has a unique minimizer $u_h^{n+1}$ on $\mathcal{M}_h^{\ast}$, and this minimizer is also the unique solution of the scheme \eqref{eq-2}, and we have,
\begin{equation}\label{eq33}
G[u_h^{n+1}]=\frac{1}{2\tau}\|u_h^{n+1}-u_h^n\|_{-1}^2+E_h(u_h^{n+1})\leq G[u_h^n]=E_h(u_h^n),
\end{equation}
this implies that $E_h(u_h^{n+1})\leq E_h(u_h^n)$.
This concludes the proof.
\end{proof}

Next, we will use the energy stability, {\L}ojasiewicz inequality \eqref{eq8} and Theorem \ref{the1} to prove that the sequence produced by the fully discrete first-order backward Euler scheme \eqref{eq-2} converges to an equilibrium as $n\to \infty$.
\begin{thm}\label{the2}
If $\tau$ satisfies condition \eqref{tau1}, every sequence $(u_h^n,\omega_h^n)_{n\geq 0}$ which complies with the scheme \eqref{eq-2} converges to a steady state $(u_h^\infty,\omega_h^\infty)$, i.e., a solution of
\begin{equation}\label{eq2222}
\begin{cases}
    (u_h^\infty\|\mathbf{1})=(u_h^0\|\mathbf{1}),\\
    \omega_h^\infty =F'(u_h^\infty)+\varepsilon^2 [J\circledast 1]u_h^\infty-\varepsilon^2[J\circledast u_h^\infty].
\end{cases}
\end{equation}
\end{thm}

The basic idea of the proof is as follows. We first use energy stability \eqref{eq33} to obtain the boundedness of the sequence $(u_h^n)_{n\geq 0}$. Next, we reformulate the scheme \eqref{eq-2} equivalently in matrix form in $\mathbb{R}^M$ and demonstrate that this reformulated scheme is a gradient flow of the energy $\tilde{\Phi}$ defined by \eqref{eq32} in an appropriate $M-1$-dimensional subspace of $\mathbb{R}^M$ by using the properties of $A_h$ and $A_J$. Finally, by applying {\L}ojasiewicz inequality \eqref{eq8} and Theorem \ref{the1}, we can derive that the sequence $(u_h^n)_{n\geq 0}$ converges to an equilibrium as $n\to \infty$.
\begin{proof}
\textbf{Step 1. $(u_h^n)_{n\geq 0}$ is bounded in $\mathcal{C}_{N\times N}^{per}$.}
Inequality \eqref{eq33} shows that $(E_h(u_h^n))_{n\geq 0}$ is nonincreasing, and
\[
E_h(u_h^n)\geq \frac{1}{2}\|u_h^n\|_2^2-\frac{3}{4}|\Omega|,
\]
i.e., $(E_h(u_h^n))_{n\geq 0}$ is bounded and so the sequence $(u_h^n)_{n\geq 0}$ is bounded in $\mathcal{C}_{N\times N}^{per}$. Therefore, there exists a subsequence $(u_h^{n_k})_{n_k\geq 0}$ such that $\lim_{n_k\to \infty}u_h^{n_k} =u_h^\infty$. Next, we will prove that the whole sequence $(u_h^n)_{n\geq 0}$ converges to $u_h^\infty$.

\textbf{Step 2. Construction of the matrix form with gradient flow structure for the scheme \eqref{eq-2}.}
By ordering the cell-centered nodes lexicographically, the first-order backward Euler scheme \eqref{eq-2} can be rephrased as: given $U^0\in \mathbb{R}^M(M=N^2)$, for $n=0,1,\cdots$, find vector-valued functions $(U^{n+1},W^{n+1})\in \mathbb{R}^M\times \mathbb{R}^M$ such that
\begin{equation}\label{eq17}
\begin{cases}
\frac{U^{n+1}-U^n}{\tau}= -A_hW^{n+1},\\
W^{n+1}=\nabla F_M(U^{n+1})+\varepsilon^2 A_J U^{n+1},
\end{cases}
\end{equation}
where $U^n=(u_i^n)_{1\leq i\leq M}$ and $W^n=(\omega_i^n)_{1\leq i\leq M}$. $A_h$ and $A_J$ are the symmetric positive semi-definite matrices defined in \eqref{eq29} and \eqref{eq28}, respectively. The nonlinear term $F_M: \mathbb{R}^M\to \mathbb{R}$ is defined by
$
F_M(v_1,\cdots,v_M)=\frac{1}{4}\sum_{i=1}^{M}((v_i)^2-1)^2.
$
It is convenient to introduce the discrete energy
\begin{equation}\label{eq20}
\Phi(U)=\frac{\varepsilon^2}{2}U^TA_JU+F_M(U).
\end{equation}

By multiplying both sides of the first equation in \eqref{eq17} by $E_1^T$ and using induction, we can deduce that $E_1^T U^n=E_1^T U^0$ for all $n\geq 0$. Since $A_h$ and $A_J$ are both symmetric positive semi-definite matrices, with $0$ being an eigenvalue of both associated with the eigenvector $E_1$, we choose an orthonormal basis of $\mathbb{R}^M$ starting with $E_1/\|E_1\|$, following the process described in Theorem 3.5 of \cite{Pierre1}. This allows us to transform matrices $A_h$ and $A_J$ into:
\begin{equation}\label{eq31}
A_h=\left(\begin{array}{ccc}0 & \cdots & 0 \\ \vdots & \tilde{A}_h & \\ 0 & & \end{array}\right),\quad
A_J=\left(\begin{array}{ccc}0 & \cdots & 0 \\ \vdots & \tilde{A}_J & \\ 0 & & \end{array}\right),
\end{equation}
where $\tilde{A}_h$ and $\tilde{A}_J$ are $(M-1)\times (M-1)$ symmetric matrices, and $\tilde{A}_h$ is positive definite. Hence, the inverse of $\tilde{A}_h$ exists and is denoted by $\tilde{A}_h^{-1}$.

Using the decomposition of \eqref{eq31}, we can deduce from \eqref{eq17} that $u_1^{n+1}=u_1^n=u_1^0$ for all $n\geq 0$. Therefore, in the subsequent proof, we only need to consider the remaining $M-1$ components of the numerical solution $U^{n+1}$. We write
$
\tilde{U}=(u_2,\cdots,u_M)^T,\;\tilde{W}=(\omega_2,\cdots,\omega_M)^T,
$
and
$
\tilde{F}_M(\tilde{U})=F_M(u_1^0,u_2,\cdots,u_M).
$
We denote
$
\tilde{\nabla} \tilde{F}_M(\tilde{U})=\left(\partial_{u_2}\tilde{F}_M(\tilde{U}),\cdots,\partial_{u_M}\tilde{F}_M(\tilde{U})\right)^T.
$
Then \eqref{eq17} becomes
\begin{equation}\label{eq19}
\begin{cases}
\frac{\tilde{U}^{n+1}-\tilde{U}^n}{\tau}= -\tilde{A}_h\tilde{W}^{n+1},\\
\tilde{W}^{n+1}=\tilde{\nabla} \tilde{F}_M(\tilde{U}^{n+1})+\varepsilon^2 \tilde{A_J} \tilde{U}^{n+1},
\end{cases}
\end{equation}
where the nonlinear term $\tilde{F}_M: \mathbb{R}^{M-1}\to \mathbb{R}$ is a nonnegative real analytic function (since it is a polynomial). The corresponding discrete energy is denoted
\begin{equation}\label{eq32}
    \tilde{\Phi}(\tilde{U})=\frac{1}{2}\tilde{U}^T\tilde{A_J} \tilde{U}+\tilde{F}_M(\tilde{U})=\Phi(u_1^0,\tilde{U}), \quad \tilde{U}\in \mathbb{R}^{M-1}.
\end{equation}

\textbf{Step 3. The energy $\tilde{\Phi}$ satisfies the {\L}ojasiewicz inequality and the sequence $(\tilde{U}^n)_{n\geq0}$ satisfies the assumptions in Theorem \ref{the1} in $\mathbb{R}^{M-1}$.}
When $\tau$ satisfies condition \eqref{tau1}, according to \eqref{eq33}, we have
\begin{align}\label{eq-13}
\tilde{\Phi}(\tilde{U}^n)-\tilde{\Phi}(\tilde{U}^{n+1})&\geq \frac{1}{2\tau}(\tilde{U}^{n+1}-\tilde{U}^n)^T\tilde{A}_h^{-1}(\tilde{U}^{n+1}-\tilde{U}^n)\nonumber\\
& \geq \frac{1}{2\tau \lambda_M}\|\tilde{U}^{n+1}-\tilde{U}^n\|^2, \quad n\geq0,
\end{align}
where $\lambda_{M}>0$ is the largest eigenvalue of $A_h$. Setting $G(U)=\tilde{\Phi}(\tilde{U})$, we find that assumption $\textbf{H1}$ is satisfied. Since $\tilde{\Phi}$ is a real analytic function in $\mathbb{R}^{M-1}$, it follows from Theorem \ref{lem3} that $\tilde{\Phi}$ satisfies the {\L}ojasiewicz inequality at every point in $\mathbb{R}^{M-1}$.

Using the discrete energy \eqref{eq32}, \eqref{eq19} can be equivalently written as
\begin{equation}\label{eq24}
    \tilde{A}_h^{-1}\left(\frac{\tilde{U}^{n+1}-\tilde{U}^n}{\tau}\right)
    =-\left(\tilde{\nabla} \tilde{F}_M(U^{n+1})+\varepsilon^2 \tilde{A}_J \tilde{U}^{n+1}\right)
    =-\tilde{\nabla} \tilde{\Phi}(\tilde{U}^{n+1}),\quad n\geq 0,
\end{equation}
which is the gradient flow of $\tilde{\Phi}$. We then obtain
\begin{equation}
\|\tilde{\nabla} \tilde{\Phi}(\tilde{U}^{n+1})\|=\left\|-\frac{1}{\tau} \tilde{A}_h^{-1}(\tilde{U}^{n+1}-\tilde{U}^n)\right\|
\leq c_3\|\tilde{U}^{n+1}-\tilde{U}^n\|,
\end{equation}
where $c_3=C(\|\tilde{A}_h^{-1}\|,\tau)$ is independent of $n$. This shows that assumption $\textbf{H2}$ holds. Let us denote $U^\infty:=(u_1^0,u_2^\infty,\cdots,u_M^\infty)^T$ and $U^n=(u_1^0,u_2^n,\cdots,u_M^n)^T$. Applying Theorem \ref{the1}, we obtain that $U^n$ converges to $U^{\infty}$ as $n\to +\infty$ and $E_1^TU^{\infty}=E_1^TU^0$. Finally, taking the limit in \eqref{eq-2}, we complete the proof.
\end{proof}

\subsection{The first-order convex splitting scheme}
In this subsection, we prove that the sequence generated by the first-order convex splitting scheme converges to an equilibrium as time tends to infinity.
We write the energy $E(u)$ as the difference of two convex functions, i.e., $E=E_c-E_e$, where
\begin{align}
    &E_c(u)=\frac{1}{4}\|u\|_{L^4}^4+\frac{1}{4}|\Omega|+\frac{2\varepsilon^2(\textrm{J}\ast 1)}{2}\|u\|_{L^2}^2,\label{eq35}\\
    &E_e(u)=\frac{1}{2}\|u\|_{L^2}^2+\frac{\varepsilon^2(\textrm{J}\ast 1)}{2}\|u\|_{L^2}^2+\frac{1}{2}(u,\textrm{J}\ast u)_{L^2}. \label{eq36}
\end{align}
We adopt the above decomposition, which allows us to separate the nonlinear term (implicitly treated) and nonlocal term (explicitly treated) without sacrificing numerical stability \cite{Guan1}.

The corresponding convex splitting of the discrete energy is given by:
$
 E_h(u_h)=E_{c,h}(u_h)-E_{e,h}(u_h),
 $
where
\begin{align}
    &E_{c,h}(u_h)=\frac{1}{4}\|u_h\|_4^4+\frac{1}{4}|\Omega|+\frac{2\varepsilon^2[J\circledast \mathbf{1}]}{2}\|u_h\|_2^2,\\
    &E_{e,h}(u_h)=\frac{1}{2}\|u_h\|_2^2+\frac{\varepsilon^2[J\circledast \mathbf{1}]}{2}\|u_h\|_2^2+\frac{1}{2}(u_h\|J\circledast u_h).
\end{align}
Then, the first-order convex splitting scheme for the N-CH equation can be formulated as \cite{Guan1}: given $u_h^0\in \mathcal{C}_{N\times N}^{per}$ and for $n\geq 0$, find $(u_h^{n+1}, \omega_h^{n+1}) \in \mathcal{C}_{N\times N}^{per}\times \mathcal{C}_{N\times N}^{per}$ such that
\begin{eqnarray}\label{eq37}
\left\{
\begin{array}{lllll}
\frac{u_h^{n+1}-u_h^n}{\tau}= \Delta_h \omega_h^{n+1}, \quad \\
\omega_h^{n+1} =(u_h^{n+1})^3+2\varepsilon^2 [J\circledast \mathbf{1}]u_h^{n+1}-u_h^n-\varepsilon^2 [J\circledast \mathbf{1}]u_h^n-\varepsilon^2[J\circledast u_h^n].\\
\end{array}
\right.
\end{eqnarray}
This is a nonlinear scheme with first order accuracy in time, which is unconditionally uniquely solvable, mass conservative and energy stable.
\begin{thm}\cite{Guan1}\label{the6}
The scheme \eqref{eq37} is unconditionally uniquely solvable and the mass is conserved, i.e., $(u_h^{n+1}-u_h^n\| \textbf{1})=0$ for all $n\geq 0$. Moreover, for any $\tau >0$,
\begin{equation}\label{eq40}
        E_h(u_h^{n+1})+\tau \|\nabla_h \omega_h^{n+1}\|_2^2+\varepsilon^2[J\circledast \mathbf{1}]\|u_h^{n+1}-u_h^n\|_2^2\leq E_h(u_h^n).
\end{equation}
\end{thm}
\begin{thm}\label{the5}
As $n\to +\infty$, the sequence $(u_h^n,\omega_h^n)_{n\geq 0}$ which generated by the scheme \eqref{eq37} converges to a steady state $(u_h^\infty,\omega_h^\infty)$, which is a solution of steady-state equation \eqref{eq2222}.
\end{thm}
\begin{proof}
\textbf{Step 1. $(u_h^n)_{n\geq 0}$ is bounded in $\mathcal{C}_{N\times N}^{per}$.}
According to \eqref{eq40}, we can obtain that $(E_h(u_h^n))_{n\geq 0}$ is nonincreasing and bounded, so $(u_h^n)_{n\geq 0}$ is bounded in $\mathcal{C}_{N\times N}^{per}$. Hence, there exists a subsequence $(u_h^{n_k})_{n_k\geq 0}$, such that $\lim_{n_k\to \infty}u_h^{n_k} =u_h^\infty$. Next, we prove that the whole sequence $(u_h^n)_{n\geq 0}$ converges to $u_h^\infty$.

\textbf{Step 2. Construction of the matrix form for the scheme \eqref{eq37}.}
By writing the first term $u_h^n$ in the second equation of \eqref{eq37} as $u_h^n=u_h^{n+1}+(u_h^n-u_h^{n+1})$, we obtain
\begin{align*}
\omega_h^{n+1} &=(u_h^{n+1})^3+2\varepsilon^2 [J\circledast \mathbf{1}]u_h^{n+1}-u_h^n-\varepsilon^2 [J\circledast \mathbf{1}]u_h^n-\varepsilon^2[J\circledast u_h^n]\\
&=(u_h^{n+1})^3-u_h^{n+1}+(2\varepsilon^2 [J\circledast \mathbf{1}]+1)(u_h^{n+1}-u_h^n)+\varepsilon^2([J\circledast \mathbf{1}]u_h^n-[J\circledast u_h^n]).
\end{align*}
Then, we can rewrite the scheme  \eqref{eq37} as follows: let $U^0\in \mathbb{R}^M$ and for $n=0,1,\cdots$, find $(U^{n+1},W^{n+1})\in \mathbb{R}^M\times \mathbb{R}^M$ solve
\begin{equation}\label{eq38}
\begin{cases}
\frac{U^{n+1}-U^n}{\tau}= -A_hW^{n+1},\\
W^{n+1}=\nabla F_M(U^{n+1})+\varepsilon^2 A_J U^n+\gamma_2(U^{n+1}-U^n),
\end{cases}
\end{equation}
where $\gamma_2=2\varepsilon^2 [J\circledast \mathbf{1}]+1$. The discrete energy is still $\Phi$ defined by \eqref{eq20}.

Similar to the argument in \textbf{Step 2} of Theorem \ref{the2}, we have $E_1^T U^n=E_1^T U^0$. Furthermore, equation \eqref{eq38} can be written as
\begin{equation}\label{eq39}
\begin{cases}
\frac{\tilde{U}^{n+1}-\tilde{U}^n}{\tau}= -\tilde{A}_h\tilde{W}^{n+1},\\
\tilde{W}^{n+1}=\tilde{\nabla} \tilde{F}_M(\tilde{U}^{n+1})+\varepsilon^2 \tilde{A}_J \tilde{U}^n+\gamma_2(\tilde{U}^{n+1}-\tilde{U}^n).
\end{cases}
\end{equation}
where the nonlinear term $\tilde{F}_M: \mathbb{R}^{M-1}\to \mathbb{R}$ is a real analytic function of class $C^2$. For ease of analysis, in the remainder of the proof, we omit the tilde symbol and replace $M-1$ with $M$. We thus continue to study \eqref{eq38}, under the assumption that $A_h$ is an $M\times M$ symmetric positive definite matrix with inverse $A_h^{-1}$.

\textbf{Step 3. The energy $\tilde{\Phi}$ satisfies the {\L}ojasiewicz inequality and the sequence $(U^n)_{n\geq0}$ satisfies the assumptions in Theorem \ref{the1} in $\mathbb{R}^{M-1}$.}
By \eqref{eq40}, we have
\begin{align*}
    \Phi(U^n)-\Phi(U^{n+1})&\geq \tau (W^{n+1})^TA_h W^{n+1}\\
    &=\frac{1}{\tau}(U^{n+1}-U^n)^T A_h^{-1}(U^{n+1}-U^n)\\
    &\geq \frac{1}{\tau \lambda_M}\|U^{n+1}-U^n\|^2, \quad \forall n\geq0,
\end{align*}
where $\lambda_{M}>0$ is the largest eigenvalue of $A_h$. Setting $G(U)=\Phi(U)$, we find that assumption $\textbf{H1}$ is satisfied. By eliminating $W^{n+1}$ from \eqref{eq38}, we obtain
\begin{align*}
    -\frac{1}{\tau}A_h^{-1}(U^{n+1}-U^n)&=\nabla F_M(U^{n+1})+\varepsilon^2 A_J U^n+\gamma_2(U^{n+1}-U^n)\\
    &=\nabla  \Phi(U^{n+1})+(\gamma_2 I-\varepsilon^2 A_J)(U^{n+1}-U^n),
\end{align*}
Thus, we find
\[
\|\nabla  \Phi(U^{n+1})\|=\left\|\left(\frac{1}{\tau}A_h^{-1}+(\gamma_2 I-\varepsilon^2 A_J)\right)(U^{n+1}-U^n)\right\|\leq C(\tau,\varepsilon,\|A_J\|,\|A_h^{-1}\|,\gamma_2)\|U^{n+1}-U^n\|,\quad \forall n\geq 0.
\]
From this, we find that $\textbf{H2}$ holds. We note that $\Phi$ is a real analytic function in $\mathbb{R}^M$, according to Theorem \ref{lem3}, $\Phi$ satisfies the {\L}ojasiewicz inequality at each point in $\mathbb{R}^M$.
Applying Theorem \ref{the1}, we obtain that the sequence $(U^n)_n$ converges to $U^\infty=(u_i^\infty)_{\leq i\leq M}$. Passing to the limit with respect to $n$ in \eqref{eq37} concludes the proof.
\end{proof}

\subsection{The first-order stabilized linear semi-implicit (SSI1) scheme}
In order to ensure that the linear scheme is energy stable, we need to make the following assumption on the nonlinear potential $F$:
\begin{Assump}\label{assum1}
    There exists a constant $\beta>0$, such that
   $
        \max_{s\in \mathbb{R}}|F''(s)|\leq \beta.
    $
\end{Assump}

Therefor, for a given $K>1$, we can replace $F(u)=\frac{1}{4}(u^2-1)^2$ by the modified potential $F_K$ \cite{Shen1} ,
\begin{equation}\label{eq411}
F_K(u)=
	\begin{cases}
 		\frac{3K^2-1}{2}u^2-2K^3u+\frac{1}{4}(3K^4+1), &u>K, \\
		\frac{1}{4}(u^2-1)^2, &u\in[-K,K], \\
        \frac{3K^2-1}{2}u^2+2K^3u+\frac{1}{4}(3K^4+1), &u<-K,
		\end{cases}
\end{equation}
Moreover,
\[
F_K'(u)=
	\begin{cases}
 		(3K^2-1)u-2K^3, &u>K, \\
		u^3-u, &u\in[-K,K], \\
        (3K^2-1)u+2K^3, &u<-K.
		\end{cases}
\]
We note that $F_K\in C^2(\mathbb{R},\mathbb{R})$ and that $F_K''(s)$ satisfies Assumption \ref{assum1}.

The SSI1 scheme of the N-CH equation is established as follows. Let $u_h^0\in \mathcal{C}_{N\times N}^{per}$, for all $n\geq 0$, find $(u_h^{n+1}, \omega_h^{n+1}) \in \mathcal{C}_{N\times N}^{per}\times \mathcal{C}_{N\times N}^{per}$ such that
\begin{eqnarray}\label{eq42}
\left\{
\begin{array}{lllll}
\frac{u_h^{n+1}-u_h^n}{\tau}= \Delta_h \omega_h^{n+1}, \hspace{0.7cm}\ \\
\omega_h^{n+1} =F_K'(u_h^n)+S(u_h^{n+1}-u_h^n)+\varepsilon^2 [J\circledast \mathbf{1}]u_h^{n+1}-\varepsilon^2[J\circledast u_h^{n+1}] ,\\
\end{array}
\right.
\end{eqnarray}
where $S>0$ is a stabilization constant. This is a linear scheme with first order accuracy and has a unique solution. The mass is conserved, which can be verified by taking the discrete $\ell^2$ inner product of the first equation in \eqref{eq42} with $\mathbf{1}$.

Define a discrete version of the energy as
\begin{equation}\label{eq48}
    E_{K,h}(u_h)=h^2(F_K(u_h)\| \mathbf{1})+\frac{\varepsilon^2 [J\circledast \mathbf{1}]}{2}\| u_h\|_2^2-\frac{\varepsilon^2}{2}h^2(u_h\| [J\circledast u_h]), \quad \forall u_h\in \mathcal{C}_{N\times N}^{per}.
\end{equation}
The energy stability result for the SSI1 scheme \eqref{eq42} is as follows:
\begin{thm}\label{the7}
For $S\geq \frac{\beta}{2}$, the SSI1 scheme \eqref{eq42} is unconditionally energy stable, i.e.,
\begin{equation}\label{eq43}
        E_{K,h}(u_h^{n+1})+\tau\|\nabla_h\omega_h^{n+1}\|_2^2\leq E_{K,h}(u_h^n),\quad n\geq 0,
    \end{equation}
\end{thm}
\begin{proof}
Taking the discrete $\ell^2$ inner product of the first and second equation in the SSI1 scheme \eqref{eq42} with $\omega_h^{n+1}$ and $u_h^{n+1}-u_h^n$, respectively, leads to
\begin{align}\label{eq45}
-\tau\|\nabla_h\omega_h^{n+1}\|_2^2&=h^2(F_K'(u_h^n)\| u_h^{n+1}-u_h^n)+S\|u_h^{n+1}-u_h^n\|_2^2\nonumber\\
&+ \varepsilon^2 h^2\left([J\circledast \mathbf{1}]u_h^{n+1}\| u_h^{n+1}-u_h^n\right)-\varepsilon^2 h^2\left([J\circledast u_h^{n+1}]\| u_h^{n+1}-u_h^n\right).
\end{align}
For the nonlinear term on the right-hand side of \eqref{eq45}, we use a Taylor expansion
\begin{equation}\label{eq46}
    h^2(F_K(u_h^{n+1})\|\mathbf{1})-h^2(F_K(u_h^n)\|\mathbf{1})=h^2\left(F_K'(u_h^n)\| u_h^{n+1}-u_h^n\right)+\frac{h^2}{2}\left(F''_K(\xi^n)(u_h^{n+1}-u_h^n)\|u_h^{n+1}-u_h^n\right),
\end{equation}
with $\xi^n$ between $u_h^{n+1}$ and $u_h^n$. For the nonlocal term on the right-hand side of \eqref{eq45}, by Lemma \ref{lem1}, we have
\begin{align}\label{eq47}
&h^2\left(([J\circledast \mathbf{1}]u_h^{n+1}-[J\circledast u_h^{n+1}])\| u_h^{n+1}-u_h^n\right)\nonumber\\
&\overset{\eqref{eq-5}}{=}\frac{[J\circledast \mathbf{1}]}{2}\|u_h^{n+1}\|_2^2-\frac{[J\circledast \mathbf{1}]}{2}\|u_h^n\|_2^2+\frac{[J\circledast \mathbf{1}]}{2}\|u_h^{n+1}-u_h^n\|_2^2\nonumber\\
&\quad-\frac{h^2}{2}([J\circledast u_h^{n+1}]\| u_h^{n+1})+\frac{h^2}{2}([J\circledast u_h^n]\| u_h^n)-\frac{h^2}{2}([J\circledast (u_h^{n+1}-u_h^n)]\| u_h^{n+1}-u_h^n)\nonumber\\
&\overset{\eqref{eq-6}}{\geq} \frac{[J\circledast \mathbf{1}]}{2}\|u_h^{n+1}\|_2^2-\frac{[J\circledast \mathbf{1}]}{2}\|u_h^n\|_2^2 -\frac{h^2}{2}([J\circledast u_h^{n+1}]\| u_h^{n+1})+\frac{h^2}{2}([J\circledast u_h^n]\| u_h^n).
\end{align}
A combination of \eqref{eq45}-\eqref{eq47} and Assumption \ref{assum1}, yield
\begin{align*}
&E_{K,h}(u_h^{n+1})+\tau\|\nabla_h\omega_h^{n+1}\|_2^2+S\|u_h^{n+1}-u_h^n\|_2^2\\
&\leq E_{K,h}(u_h^n)+\frac{h^2}{2}(F''_K(\xi^n)(u_h^{n+1}-u_h^n)\|u_h^{n+1}-u_h^n)\\
&\leq E_{K,h}(u_h^n)+\frac{\beta}{2}\|u_h^{n+1}-u_h^n\|_2^2.
\end{align*}
then,
\begin{equation}
    E_{K,h}(u_h^{n+1})+\tau\|\nabla_h\omega_h^{n+1}\|_2^2+\left(S-\frac{\beta}{2}\right)\|u_h^{n+1}-u_h^n\|_2^2\leq E_{K,h}(u_h^n).
\end{equation}
 So, if $S\geq \frac{\beta}{2}$, we obtain the result \eqref{eq43}.
\end{proof}

We now prove that the whole sequence generated by the SSI1 scheme \eqref{eq42} converges to an equilibrium. Firstly, the  SSI1 scheme \eqref{eq42} can be written in matrix form: given $U^0\in \mathbb{R}^M$, for $n=0,1,\cdots$, find $(U^{n+1},W^{n+1})\in \mathbb{R}^M\times \mathbb{R}^M$ solve
\begin{equation}\label{eq49}
\begin{cases}
\frac{U^{n+1}-U^n}{\tau}= -A_hW^{n+1},\\
W^{n+1}=\nabla F_{K,M}(U^n)+\varepsilon^2 A_J U^{n+1}+S(U^{n+1}-U^n),
\end{cases}
\end{equation}
where $F_{K,M}:\mathbb{R}^M\to \mathbb{R}$ is defined by
$
F_{K,M}(v_1,\cdots,v_M)=\sum_{i=1}^{M}F_K(v_i).
$
The corresponding discrete energy is defined by
\begin{equation}\label{eq50}
\Phi_K(U)=\frac{\varepsilon^2}{2}U^TA_JU+F_{K,M}(U).
\end{equation}

For the linear SSI1 scheme, the nonlinearity is no longer analytic but it is semi-algebraic \cite{ABF,BR}. Kurdyka proved that the {\L}ojasiewicz inequality also holds for semi-algebraic functions \cite{Kurdyka1}. The main results is as follows:

\begin{thm}\cite{Kurdyka1}\label{the8}
If $G\in C^1(\mathbb{R}^M,\mathbb{R})$ is semi-algebraic, then $G$ satisfies {\L}ojasiewicz inequality at each point in $\mathbb{R}^M$.
\end{thm}

\begin{lem}\cite{Pierre1}\label{lem6}
The nonlinearity $F_{K,M}\in C^2(\mathbb{R}^M, \mathbb{R})$ in \eqref{eq49} is a nonnegative semi-algebraic function of class $C^2$.
\end{lem}

The convergence result is presented as follows.

\begin{thm}
If $S\geq \frac{\beta}{2}$, the sequence $(u_h^n,\omega_h^n)_{n\geq 0}$ which complied by \eqref{eq42} converges to a steady state $(u_h^\infty,\omega_h^\infty)$, i.e., a solution of the equation \eqref{eq2222}.
\end{thm}
\begin{proof}
Following the arguments in \textbf{Step 1} and \textbf{Step 2} of Theorem \ref{the2}, we can show that the sequence $(U^n)_{n\geq 0}$ is bounded in $\mathbb{R}^M$. Furthermore, we may assume that the matrix $A_h$ in \eqref{eq49} is  symmetric positive definite matrix and invertible. By Lemma \ref{lem6}, the function $\Phi_K$ satisfies the {\L}ojasiewicz inequality at each point in $\mathbb{R}^M$.

The next step is to verify that the sequence $(U^n)_{n\geq 0}$ satisfies the assumptions of Theorem \ref{the1}.
According to Theorem \ref{the7}, the following holds
\begin{align*}
    \Phi_K(U^n)- \Phi_K(U^{n+1})&\geq \tau (W^{n+1})^T A_h W^{n+1}\\
    &= \frac{1}{\tau}(U^{n+1}-U^n)^T A_h^{-1}(U^{n+1}-U^n)\\
    &\geq \frac{1}{\tau \lambda_M}\|U^{n+1}-U^n\|^2, \quad \forall n\geq 0,
    \end{align*}
where $\lambda_{M}>0$ is the largest eigenvalue of $A_h$. Let $G(U)=\Phi_K(U)$. Then assumption $\textbf{H1}$ is satisfied. Next, we rewrite \eqref{eq49} as
\begin{align*}
        -\frac{1}{\tau}A_h^{-1}(U^{n+1}-U^n)&=\nabla F_{K,M}(U^n)+\varepsilon^2 A_J U^{n+1}+S(U^{n+1}-U^n)\\
        &=\nabla F_{K,M}(U^n)-\nabla F_{K,M}(U^{n+1})+\nabla \Phi_K(U^{n+1})+S(U^{n+1}-U^n),
\end{align*}
and using the triangle inequality, we obtain
\begin{align*}
    \|\nabla \Phi_K(U^{n+1})\|&\leq \frac{1}{\tau}\|A_h^{-1}(U^{n+1}-U^n)\|+(\beta+S) \|(U^{n+1}-U^n)\|\\
    &\leq C(\tau,\|A_h^{-1}\|,\beta,S)\|(U^{n+1}-U^n)\|, \quad \forall n\geq 0,
\end{align*}
where $C>0$ is a constant independent of $n$. This implies that assumption $\textbf{H2}$ holds. Applying Theorem \ref{the1} ~completes the proof.
\end{proof}
\begin{rmk}
For the three energy-stable first-order fully discrete schemes mentioned above, application of the {\L}ojasiewicz inequality shows that a local minimizer of the discrete energy is stable with respect to the time step size $\tau$. For detailed proofs, we refer the interested reader to \cite{Pierre2,Pierre3}.
\end{rmk}
\section{Convergence to equilibrium for the second-order fully discrete schemes}\label{sec4}
This section focus on two energy-stable second-order fully discrete schemes for the N-CH equation, including the BDF2 scheme and the linearly implicit (2LI) scheme. First, we construct a fully discrete BDF2 scheme and establish its unique solvable and modified energy stability. Then, by combining the energy stability with the {\L}ojasiewicz inequality \eqref{eq8} and utilizing Theorem \ref{the1}, we rigorously prove that the sequence generated by the BDF2 scheme converges to a steady state as time goes to infinity. Finally, we introduce a fully discrete 2LI scheme and prove its discrete modified energy stability. The convergence analysis for this scheme follows  similar arguments. Therefore, we only outline the main idea and state the convergence result.

\subsection{BDF2 scheme}\label{BDF2}
The BDF2 scheme for the N-CH equation reads: given $u_h^0,u_h^1\in \mathcal{C}_{N\times N}^{per}$, for all $n\geq 1$, find $(u_h^{n+1}, \omega_h^{n+1}) \in \mathcal{C}_{N\times N}^{per}\times \mathcal{C}_{N\times N}^{per}$ such that
\begin{eqnarray}\label{equ1}
\left\{
\begin{array}{lllll}
\frac{3u_h^{n+1}-4u_h^n+u_h^{n-1}}{2\tau}= \Delta_h \omega_h^{n+1}, \hspace{0.7cm}\ \\
\omega_h^{n+1} =F'(u_h^{n+1})+\varepsilon^2 [J\circledast \mathbf{1}]u_h^{n+1}-\varepsilon^2[J\circledast u_h^{n+1}].\\
\end{array}
\right.
\end{eqnarray}
We assume that $(u_h^1\|\mathbf{1})=(u_h^0\|\mathbf{1})$. Thus, the mass is conserved, which can be verified by taking the discrete inner product of both sides of the first equation in \eqref{equ1} with $\mathbf{1}$; i.e., $(u_h^{n+1}\| \mathbf{1})=(u_h^n\| \mathbf{1})=(u_h^0\|\mathbf{1})$ for all $n\geq 0$. The discrete energy is still $E_h(u_h)$, as defined in \eqref{1}.

For notational simplicity, we denote the difference $\delta u_h^n=u^n-u^{n-1}$ and define the following relation:
\begin{equation}\label{equ2}
    3u_h^{n+1}-4u_h^n+u_h^{n-1}=2\delta u_h^{n+1}+(\delta u_h^{n+1}-\delta u_h^n).
\end{equation}
We first prove the unique solvability of the BDF2 scheme \eqref{equ1}.

\begin{thm}\label{Th1}
If $\tau$ satisfies condition \eqref{tau1}, then, the BDF2 scheme \eqref{equ1} has a unique solution.
\end{thm}
\begin{proof}
We consider the energy functional $G[z]$ on the space $\mathcal{M}_h^{\ast}=\{z\in \mathcal{C}_{N\times N}^{per}| \; (z\|\mathbf{1} )= ( u_h^n \| \mathbf{1})\}$,
\[
    G[z]=\frac{3}{4\tau}\|z-u_h^n\|_{-1}^2+\frac{h^2}{2\tau}\left((-\Delta_h)^{-1}(u_h^{n-1}-u_h^n)\|z-u_h^n\right)+E_h(z).
\]
Using Lemma \ref{lem1} and Young's inequality, we have
\begin{align*}
    G[z] &\geq-\frac{1}{12\tau}\|u_h^{n-1}-u_h^n\|_{-1}^2+\frac{\varepsilon^2[J\circledast \mathbf{1}]}{2}\|z\|_2^2-\frac{\varepsilon^2 h^2}{2}(z\| [J\circledast z])+\frac{1}{4}\|z\|_4^4-\frac{1}{2}\|z\|_2^2+\frac{1}{4}|\Omega|\\
    &\geq -\frac{1}{12\tau}\|u_h^{n-1}-u_h^n\|_{-1}^2+\frac{1}{4}(4\|z\|_2^2-4|\Omega|)-\frac{1}{2}\|z\|_2^2+\frac{1}{4}|\Omega|\\
    &=-\frac{1}{12\tau}\|u_h^{n-1}-u_h^n\|_{-1}^2+\frac{1}{2}\|z\|_2^2-\frac{3}{4} |\Omega|.
\end{align*}
This indicates that the functional $G[z]$ is coercive and has a minimizer. It remains to demonstrate that for any $s\in \mathbb{R}$ and $\psi \in \mathcal{M}_h^0$, $G[z]$ is strictly convex. It is easy to obtain
\begin{equation*}
\frac{\diff^2 G}{\diff s^2}[z+s\psi]\big|_{s=0}=\frac{3}{2\tau}\|\psi\|_{-1}^2+\varepsilon^2 [J\circledast\mathbf{1}]\|\psi\|^2_2-\varepsilon^2 h^2([J\circledast \psi]\|\psi)+3\|z\psi\|_2^2-\|\psi\|_2^2.
\end{equation*}
If $\tau$ satisfies condition \eqref{tau1}, by using \eqref{eq4}, we obtain
\begin{align*}
\frac{\diff^2 G}{\diff s^2}[z+s\psi]\big|_{s=0}&\geq \frac{1}{\tau}\|\psi\|_{-1}^2+\left(\varepsilon^2 [J\circledast \mathbf{1}]-\frac{C_\textrm{J}\tau\varepsilon^4}{2}-1\right)\|\psi\|_2^2+3\|z\psi\|_2^2\\
&=\frac{1}{\tau}\|\psi\|_{-1}^2+\left(\gamma_0-\frac{C_\textrm{J}\tau\varepsilon^4}{2}\right)\|\psi\|_2^2+3\|z\psi\|_2^2>0.
\end{align*}
This implies that the functional $G[z]$ has a unique minimizer $u_h^{n+1}$, which is the unique solution of the BDF2 scheme \eqref{equ1}.
\end{proof}

We now shows that the BDF2 scheme \eqref{equ1} satisfies the modified energy stability.
\begin{thm}\label{Th2}
Suppose that the time step $\tau$ satisfies condition \eqref{tau1}. Then, for the sequence $(u_h^n)_{n\geq 0}$ complies with the proposed scheme \eqref{equ1}, the following holds:
\begin{equation}\label{equ9}
\widetilde{E}_h(u_h^{n+1},\delta u_h^{n+1})\leq \widetilde{E}_h(u_h^n,\delta u_h^n),
\end{equation}
where
$
    \widetilde{E}_h(u_h,v_h)=E_h(u_h)+\frac{1}{4\tau}\| v_h\|_{-1}^2.
$
\end{thm}

\begin{proof}
We take the discrete $\ell^2$ inner product of the first equality in \eqref{equ1} with $(-\Delta_h)^{-1}\delta u_h^{n+1}$ and of the second equality with $\delta u_h^{n+1}$. Using \eqref{equ2} then yields
 \begin{align}\label{equ5}
        &\frac{1}{\tau}\|\delta u_h^{n+1}\|_{-1}^2+\frac{1}{2\tau}h^2((-\Delta_h)^{-1}(\delta u_h^{n+1}-\delta u_h^n)\|\delta u_h^{n+1})\nonumber\\
        &+h^2(F'(u_h^{n+1})\| \delta u_h^{n+1})+\varepsilon^2 [J\circledast \mathbf{1}]h^2(u_h^{n+1}\|\delta u_h^{n+1})-\varepsilon^2h^2([J\circledast u_h^{n+1}]\|\delta u_h^{n+1})=0
 \end{align}
 For the second term on the left-hand side of \eqref{equ5}, a direct calculation gives
 \begin{equation}\label{equ6}
    \frac{1}{2\tau}h^2\left((-\Delta_h)^{-1}(\delta u_h^{n+1}-\delta u_h^n)\|\delta u_h^{n+1}\right)=
    \frac{1}{4\tau}\|\delta u_h^{n+1}\|_{-1}^2
    -\frac{1}{4\tau}\|\delta u_h^n\|_{-1}^2
    +\frac{1}{4\tau}\|\delta u_h^{n+1}-\delta u_h^n\|_{-1}^2.
 \end{equation}
 For the nonlinearity term in the equality \eqref{equ5}, a Taylor expansion gives
 \begin{equation}\label{equ7}
        h^2(F'(u_h^{n+1})\| \delta u_h^{n+1})\geq h^2(F(u_h^{n+1})\| \mathbf{1})-h^2(F(u_h^n)\| \mathbf{1})-\frac{1}{2}\|\delta u_h^{n+1}\|_2^2.
 \end{equation}
  By Lemma \ref{lem1} and estimation \eqref{eq4}, the nonlocal term in \eqref{equ5} becomes
 \begin{align}\label{equ8}
        &\varepsilon^2 [J\circledast \mathbf{1}]h^2(u_h^{n+1}\|\delta u_h^{n+1})-\varepsilon^2h^2([J\circledast u_h^{n+1}]\|\delta u_h^{n+1})\nonumber \\
        &\geq \frac{\varepsilon^2}{2}[J\circledast \mathbf{1}]\|u_h^{n+1}\|_2^2-\frac{\varepsilon^2}{2}[J\circledast \mathbf{1}]\|u_h^n\|_2^2+\frac{\varepsilon^2}{2}[J\circledast \mathbf{1}]\|\delta u_h^{n+1}\|_2^2\nonumber \\
        &-\frac{\varepsilon^2}{2}h^2([J\circledast u_h^{n+1}]\|u_h^{n+1})+\frac{\varepsilon^2}{2}h^2([J\circledast u_h^n]\|u_h^n)-\frac{\varepsilon^4 C_J \tau}{4}\|\delta u_h^{n+1}\|_2^2-\frac{1}{4\tau}\|\delta u_h^{n+1}\|_{-1}^2.
 \end{align}
 Combining \eqref{equ6}-\eqref{equ8}, we deduce
 \begin{equation}\label{equ10}
    \frac{3}{4\tau}\|\delta u_h^{n+1}\|_{-1}^2+\left(\frac{\gamma_0}{2}-\frac{\varepsilon^4 C_J \tau}{4}\right)\|\delta u_h^{n+1}\|_2^2+\frac{1}{4\tau}\|\delta u_h^{n+1}-\delta u_h^n\|_{-1}^2+\widetilde{E}_h(u_h^{n+1},\delta u_h^{n+1})\leq \widetilde{E}_h(u_h^n,\delta u_h^n).
 \end{equation}
 If $\tau$ satisfies condition \eqref{tau1}, we obtain the modified energy stability \eqref{equ9}.
\end{proof}

Now, we prove that the sequence generated by the BDF2 scheme \eqref{equ1} converges to a steady state as time goes to infinity. This proof follows a similar argument to that of Theorem \ref{the2},
but we must verify that the modified energy $\tilde{\Phi}$ satisfies the {\L}ojasiewicz inequality and the assumptions of Theorem \ref{the1} in $\mathbb{R}^{2M}$.
\begin{thm}\label{Th3}
If $\tau< \frac{2\gamma_0}{C_\textrm{J}\varepsilon^4}$, then as $n\to \infty$, the sequence $(u_h^n,\omega_h^n)_{n\geq 0}$ generated by the scheme \eqref{equ1} converges to a steady state $(u_h^\infty,\omega_h^\infty)$, which is a solution of the equation \eqref{eq2222}.
\end{thm}
\begin{proof}
\textbf{Step 1. $(u_h^n)_{n\geq 0}$ is bounded in $\mathcal{C}_{N\times N}^{per}$.}
The energy inequality \eqref{equ10} implies that the sequence $(E_h(u_h^n))_{n\geq 0}$ is nonincreasing. Moreover, it satisfies
\[
\widetilde{E}_h(u_h^n,\delta u_h^n)\geq E_h(u_h^n)\geq \frac{1}{2}\|u_h^n\|_2^2-\frac{3}{4}|\Omega|,
\]
which means $(E_h(u_h^n))_{n\geq 0}$ is bounded from below. Therefore, the sequence converges. Consequently, the sequence $(u_h^n)_{n\geq 0}$ is bounded in $\mathcal{C}_{N\times N}^{per}$. Therefore, there exists a convergent subsequence $(u_h^{n_k})_{n_k\geq 0}$ such that $\lim_{n_k\to \infty}u_h^{n_k} =u_h^\infty$.

Analogous to the proof of the \textbf{Step 2} in Theorem \ref{the2}, we can write the BDF2 scheme \eqref{equ1} in matrix form: let $U^0,U^1\in \mathbb{R}^M$, and for $n=1,2,\cdots$, let $(U^{n+1},W^{n+1})\in \mathbb{R}^M\times \mathbb{R}^M$ solve
\begin{equation}\label{equ11}
\begin{cases}
\frac{3U^{n+1}-4U^n+U^{n-1}}{2\tau}= -A_hW^{n+1},\\
W^{n+1}=\nabla F_M(U^{n+1})+\varepsilon^2 A_J U^{n+1},
\end{cases}
\end{equation}
where $U^n=(u_i^n)_{1\leq i\leq M}$ and $W^n=(\omega_i^n)_{1\leq i\leq M}$. We may assume that $A_h$ is an $M\times M$ symmetric positive definite matrix with inverse $A_h^{-1}$. The discrete modified energy is defined by
\begin{equation}\label{equ12}
\widetilde{\Phi}(U,V)=\Phi(U)+\frac{1}{4\tau}V^TA_h^{-1}V,
\end{equation}
where $\Phi(U)$ is defined in \eqref{eq20}. We will later show that the modified energy $\widetilde{\Phi}$ can be recovered to $\Phi$ as $n\to \infty$.

\textbf{Step 2. The sequence $(U^n,\delta U^n)_{n\geq 0}$ satisfies the assumptions of Theorem \ref{the1} in $\mathbb{R}^{2M}$.}
We denote $\delta U^{n+1}=U^{n+1}-U^n$, and write
\begin{equation}\label{equ13}
    3U^{n+1}-4U^n+U^{n-1}=2\delta U^{n+1}+(\delta U^{n+1}-\delta U^n).
\end{equation}
By Theorem \ref{Th2}, when $\tau< \frac{2\gamma_0}{C_\textrm{J}\varepsilon^4}$, we have
\begin{equation}\label{equ14}
        \frac{c_{\tau}}{\tau}\left(\|\delta U^{n+1}\|^2+(\delta U^{n+1}-\delta U^n)^TA_h^{-1}(\delta U^{n+1}-\delta U^n)\right)+\widetilde{\Phi}(U^{n+1},\delta U^{n+1})\leq \widetilde{\Phi}(U^n,\delta U^n), \quad \forall n\geq 0,
\end{equation}
where $c_{\tau}=\min\left(\frac{\gamma_0\tau}{2}-\frac{\varepsilon^4C_J\tau^2}{4},\frac{1}{4}\right)>0$. Let $G=\widetilde{\Phi}$. Then the assumption $\textbf{H1}$ holds. From inequality \eqref{equ14}, it follows that the sequence $(\widetilde{\Phi}(U^n,\delta U^n))_{n\geq 0}$ converges to some $\Phi^{\ast}$ and $\delta U^n\to 0$ as $n\to \infty$. This implies that the original energy sequence $(\Phi(U^n))_{n\geq 0}$ converges to $\Phi^{\ast}$. Therefore, the modified energy $\widetilde{\Phi}$ and the original energy $\Phi$ share the same limit $\Phi^{\ast}$, meaning the the modified energy is recovered at infinity.
Next, we verity that $\widetilde{\Phi}$ satisfies condition $\textbf{H2}$.

By eliminating the intermediate variable $W^{n+1}$ from the scheme \eqref{equ11} and utilizing the identity \eqref{equ13}, we obtain
\begin{equation}\label{equ15}
        \|\nabla \Phi(U^{n+1})\|\leq \frac{1}{\tau}\|A_h^{-1}\|(\|U^{n+1}-U^n\|+\|\delta U^{n+1}-\delta U^n\|).
\end{equation}
Since $\nabla \widetilde{\Phi}(U,V)=\left(\nabla\Phi(U),\frac{A_h^{-1}V}{2\tau}\right)\in \mathbb{R}^{2M}$, we have
\begin{align}\label{equ16}
    \|\nabla \widetilde{\Phi}(U,V)\|^2&=\|\nabla\Phi(U)\|^2+\frac{1}{4\tau^2}\|A_h^{-1}V\|^2\nonumber\\
    &\leq \|\nabla\Phi(U)\|^2+ \frac{\|A_h^{-1}\|^2}{4\tau^2}\|V\|^2.
\end{align}
Together with \eqref{equ14}-\eqref{equ16}, we deduce
\begin{align}\label{equ17}
    \widetilde{\Phi}(U^n,\delta U^n)-\widetilde{\Phi}(U^{n+1},\delta U^{n+1})&\geq \frac{c_{\tau}\min\left(1,\frac{1}{\lambda_M}\right)}{\tau}(\|U^{n+1}-U^n\|^2+\|\delta U^{n+1}-\delta U^n\|^2)\nonumber\\
    &\geq \frac{\tau c_{\tau}\min\left(1,\frac{1}{\lambda_M}\right)}{4\|A_h^{-1}\|^2}\|\nabla \Phi(U^{n+1})\|^2+\frac{c_{\tau}\min\left(1,\frac{1}{\lambda_M}\right)}{2\tau}\|U^{n+1}-U^n\|^2\nonumber\\
    &\geq \frac{\tau c_{\tau}\min\left(1,\frac{1}{\lambda_M}\right)}{4\|A_h^{-1}\|^2}\|\nabla \widetilde{\Phi}(U^{n+1},\delta U^{n+1})\|^2,\quad \forall n\geq 0,
\end{align}
where $\lambda_{M}>0$ is the largest eigenvalue of $A_h$. This implies that $\widetilde{\Phi}$ satisfies assumption $\textbf{H2}$.

\textbf{Step 3. The modified energy $\widetilde{\Phi}$ satisfies the {\L}ojasiewicz inequality in $\mathbb{R}^{2M}$.}
From \textbf{Step 1}, we have a convergent subsequence $(U^{n_k})_{n_k\geq 0}$ of $(U^n)_{n\geq 0}$ converges to some $U^{\infty}$. Since $\Phi$ is real analytic, it satisfies the {\L}ojasiewicz inequality \eqref{eq8} at $U^{\infty}$. In the following proof, suppose $\|U-U^{\infty}\|<\sigma$ and $\|V^T A_h^{-1}V\|< 1$. Using the inequality $(a+b)^{1-\nu}\leq a^{1-\nu}+b^{1-\nu}$ for $a,b\geq 0$, we obtain
\begin{align}\label{equ18}
        |\widetilde{\Phi}(U,V)-\widetilde{\Phi}(U^{\infty},\textbf{0})|^{1-\nu} &\leq |\Phi(U)-\Phi^{\ast}|^{1-\nu}+\frac{1}{(4\tau)^{1-\nu}}(V^T A_h^{-1}V)^{\frac{2(1-\nu)}{2}}\nonumber\\
        &\leq \gamma_1\|\nabla \Phi(U)\|+\frac{1}{(4\tau)^{1-\nu}}(V^T A_h^{-1}V)^{\frac{1}{2}}\nonumber\\
        &\leq \sqrt{2}\max\left\{\gamma_1,\frac{(4\tau)^\theta \lambda_M}{2\sqrt{\lambda_1}}\right\}\left(\|\nabla \Phi(U)\|^2+\frac{1}{4\tau^2}\|A_h^{-1}V\|^2\right)^{\frac{1}{2}}\nonumber\\
        &\leq \sqrt{2}\max\left\{\gamma_1,\frac{(4\tau)^\theta \lambda_M}{2\sqrt{\lambda_1}}\right\}\|\nabla\widetilde{\Phi}(U,V)\|,
\end{align}
where $\lambda_1>0$ and $\lambda_M>0$ are the smallest and largest eigenvalues of the positive definite matrix $A_h$. Therefore, $\widetilde{\Phi}$ satisfies the {\L}ojasiewicz inequality at $(U^{\infty},\textbf{0})$. Finally, taking the limit in \eqref{equ1} and invoking Theorem \ref{the1} completes the proof.
\end{proof}

\subsection{The second-order linearly implicit (2LI) scheme}
We replace the potential function $F$ with the modified potential function $F_K$ defined by \eqref{eq411}, thereby ensuring the energy stability of the linear scheme. The 2LI scheme for the N-CH equation reads: given $u_h^0,u_h^1\in \mathcal{C}_{N\times N}^{per}$, for all $n\geq 1$, let $u_h^{n+1}, \omega_h^{n+1} \in \mathcal{C}_{N\times N}^{per}$ solve
\begin{eqnarray}\label{e1}
\left\{
\begin{array}{lllll}
\frac{3u_h^{n+1}-4u_h^n+u_h^{n-1}}{2\tau}= \Delta_h \omega_h^{n+1}, \hspace{0.7cm}\ \\
\omega_h^{n+1} =2F_K'(u_h^n)-F_K'(u_h^{n-1})+\varepsilon^2 [J\circledast \mathbf{1}]u_h^{n+1}-\varepsilon^2[J\circledast u_h^{n+1}],  \\
\end{array}
\right.
\end{eqnarray}
where the nonlinearity $F_K$ satisfies the assumption \eqref{assum1}. The 2LI scheme \eqref{e1} is a linear second-order scheme and its solution is unique. Under the condition that $(u_h^1\|\mathbf{1})=(u_h^0\|\mathbf{1})$, the 2LI scheme \eqref{e1} is mass conservative. Next, we establish its modified energy stability.
\begin{thm}\label{Th4}
If $\beta$ and $\tau$ satisfy:
\begin{equation}\label{tau2}
    \beta\leq \frac{\gamma_0+1}{3},\;\tau\leq  \frac{2(\gamma_0+1-3\beta)}{C_\textrm{J}\varepsilon^4},
\end{equation}
where $C_\textrm{J}>0$ depends on kernel function $\textrm{J}$ but is independent of $h$. Then, the 2LI scheme \eqref{e1} satisfies the modified energy stability, that is
\begin{equation}\label{equ200}
\widetilde{E}_{K,h}(u_h^{n+1},\delta u_h^{n+1})\leq \widetilde{E}_{K,h}(u_h^n,\delta u_h^n),\quad \forall n\geq 0,
\end{equation}
where
$
    \widetilde{E}_{K,h}(u_h,v_h)=E_{K,h}(u_h,v_h)+\frac{\beta}{2}\|v_h\|_2^2+\frac{1}{4\tau}\|v_h\|_{-1}^2.
$
\end{thm}

\begin{proof}
 We Take the discrete $\ell^2$ inner product of the first equation in \eqref{e1} with $(-\Delta_h)^{-1}\delta u_h^{n+1}$ and of the second equation with $\delta u_h^{n+1}$. Applying the identity \eqref{equ2} then yields
 \begin{align}\label{equ21}
        &\frac{1}{\tau}\|\delta u_h^{n+1}\|_{-1}^2+\frac{1}{2\tau}h^2(\delta u_h^{n+1}-\delta u_h^n\|(-\Delta_h)^{-1}\delta u_h^{n+1})+\varepsilon^2 [J\circledast \mathbf{1}]h^2(u_h^{n+1}\|\delta u_h^{n+1})-\varepsilon^2h^2([J\circledast u_h^{n+1}]\|\delta u_h^{n+1})\nonumber\\
        &=h^2(2F_K'(u_h^n)-F_K'(u_h^{n-1})\|u_h^n-u_h^{n+1}).
 \end{align}
 Since $F_k$ satisfies assumption \eqref{assum1}, we can use Taylor's expansion and the mean value theorem to obtain
 \begin{equation}\label{equ22}
        h^2(F_K'(u_h^n)\|u_h^n-u_h^{n+1})\leq h^2(F_K(u_h^n)\|\mathbf{1})-h^2(F_K(u_h^{n+1})\|\mathbf{1})+\frac{\beta}{2}\|u_h^n-u_h^{n+1}\|^2_2,
 \end{equation}
 and
 \begin{equation}\label{equ23}
        h^2(F_K'(u_h^n)-F_K'(u_h^{n-1})\|u_h^n-u_h^{n+1})\leq \frac{\beta}{2} \|u_h^n-u_h^{n-1}\|^2_2+\frac{\beta}{2} \|u_h^n-u_h^{n+1}\|^2_2.
 \end{equation}
 Starting from the identity \eqref{equ21}, we combine \eqref{equ22} and \eqref{equ23}, using \eqref{equ6} and \eqref{equ8} to deduce
 \begin{equation}\label{energy1}
       \frac{3}{4\tau}\|\delta u_h^{n+1}\|_{-1}^2+\left(\frac{\gamma_0+1-3\beta}{2}-\frac{\varepsilon^4 C_J \tau}{4}\right)\|\delta u_h^{n+1}\|_2^2+\frac{1}{4\tau}\|\delta u_h^{n+1}-\delta u_h^n\|_{-1}^2+\widetilde{E}_{K,h}(u_h^{n+1},\delta u_h^{n+1})\leq \widetilde{E}_{K,h}(u_h^n,\delta u_h^n).
 \end{equation}
 If the condition \eqref{tau2} holds, we conclude that the 2LI scheme satisfies the energy estimate \eqref{equ200}.
\end{proof}

We now prove the convergence to equilibrium for the 2LI scheme \eqref{e1}.
\begin{thm}
Suppose $\beta<\frac{\gamma_0+1}{3},\;\tau<\frac{2(\gamma_0+1-3\beta)}{C_\textrm{J}\varepsilon^4}$ and let $(u_h^n,\omega_h^n)_{n\geq 0}$ be a sequence which complies with \eqref{e1}. Then the whole sequence converges to a steady state $(u_h^\infty,\omega_h^\infty)$ that is a solution of \eqref{eq2222}.
\end{thm}

\begin{proof}
Analogously to \textbf{Step 2} of Theorem \ref{the2}, the 2LI scheme \eqref{e1} can be rewritten in matrix form: let $U^0,U^1\in \mathbb{R}^M$, and for $n=1,2,\cdots$, let $(U^{n+1},W^{n+1})\in \mathbb{R}^M\times \mathbb{R}^M$ solve
\begin{equation}\label{2LI1}
\begin{cases}
\frac{3U^{n+1}-4U^n+U^{n-1}}{2\tau}= -A_hW^{n+1},\\
W^{n+1}=\nabla \left(2F_{K,M}(U^n)-F_{K,M}(U^{n-1})\right)+\varepsilon^2 A_J U^{n+1},
\end{cases}
\end{equation}
where $F_{K,M}:\mathbb{R}^M\to \mathbb{R}$ is defined by $F_{K,M}(v_1,\cdots,v_M)=\sum_{i=1}^{M}F_K(v_i)$.
 We can assume that $A_h$ is an $M\times M$ symmetric positive definite matrix with inverse $A_h^{-1}$. The corresponding discrete energy is defined as
\[
\widetilde{\Phi}_K(U,V)={\Phi}_K(U)+\frac{\beta}{2}\|V\|^2+\frac{1}{4\tau}V^TA_h^{-1}V,
\]
where ${\Phi}_K(U)$ is defined by \eqref{eq50} and is semi-algebraic.

By \eqref{energy1}, when $\beta<\frac{\gamma_0+1}{3},\;\tau<\frac{2(\gamma_0+1-3\beta)}{C_\textrm{J}\varepsilon^4}$, we have
\begin{equation}\label{2LI2}
        \frac{c_{\tau}}{\tau}\left(\|U^{n+1}-U^n\|^2+(\delta U^{n+1}-\delta U^n)^TA_h^{-1}(\delta U^{n+1}-\delta U^n)\right)+\widetilde{\Phi}(U^{n+1},\delta U^{n+1})\leq \widetilde{\Phi}(U^n,\delta U^n), \quad \forall n\geq 0,
\end{equation}
where $c_{\tau}=\min\left(\frac{(\gamma_0+1-3\beta)\tau}{2}-\frac{\varepsilon^4C_J\tau^2}{4},\frac{1}{4}\right)>0$. Let $G=\widetilde{\Phi}$. Thus, assumption $\textbf{H1}$ holds. Furthermore, from inequality \eqref{2LI2}, the sequence $(U^n)_{n\geq 0}$ is bounded in $\mathbb{R}^M$. Hence, a subsequence $(U^{n_k})_{n_k\geq 0}$ converges to some $U^{\infty}$ in $\mathbb{R}^M$.

We use the scheme \eqref{2LI1}, identity \eqref{equ13} and assumption \eqref{assum1} and obtain
\begin{equation}\label{2LI3}
\frac{\|A_h^{-1}\|}{2\tau}(2\|\delta U^{n+1}\|+\|\delta U^{n+1}-\delta U^n\|)\geq \|\nabla \Phi_K(U^{n+1})\|-2\beta\|\delta U^{n+1} \|-\beta \|U^{n+1}-U^{n-1}\|.
\end{equation}
By the triangle inequality
\[
\|U^{n+1}-U^{n-1}\|=\|2\delta U^{n+1}+(\delta U^n-\delta U^{n+1})\|\leq 2\|\delta U^{n+1}\|+\|\delta U^{n+1}-\delta U^n\|,
\]
we can obtain
\begin{equation}\label{2LI4}
(\|A_h^{-1}\|+4\tau\beta)\left(\|\delta U^{n+1}\|+\|\delta U^{n+1}-\delta U^n\|\right)\geq \tau \|\nabla \Phi_K(U^{n+1})\|.
\end{equation}
We note that $\nabla \widetilde{\Phi}_K(U,V)=\left(\nabla\Phi_K(U),\left(\frac{A_h^{-1}}{2\tau}+\beta I\right)V\right)\in \mathbb{R}^{2M}$, and
\begin{align}\label{2LI5}
    \|\nabla \widetilde{\Phi}_K(U,V)\|^2&=\|\nabla\Phi_K(U)\|^2+\left\|\left(\frac{A_h^{-1}}{2\tau}+\beta I\right)V\right\|^2\nonumber\\
    &\leq \|\nabla\Phi_K(U)\|^2+ \left(\frac{\|A_h^{-1}\|}{2\tau}+\beta \right)^2\|V\|^2.
\end{align}
Combining this inequality with \eqref{2LI2} and \eqref{2LI4}, we derive that the sequence $(U^n,\delta U^n)_{n\geq 0}$ also satisfies condition \textbf{H2}. As demonstrated in the \textbf{Step 3} of Theorem \ref{Th3}, we obtain the energy $\widetilde{\Phi}_K$ satisfies the {\L}ojasiewicz inequality near $(U^{\infty},\textbf{0})$. Applying Theorem \ref{the1}, we conclude the proof.
\end{proof}
\begin{rmk}
For the two energy-stable second-order fully discrete schemes considered above, because the modified energy functions depend on $\tau$, the stability results for first-order schemes in \cite{Pierre2,Pierre3} cannot be directly applied to analyze the stability of a local minimizer as $\tau\to 0$. A proof of this stability will be provided in our future works.
\end{rmk}

\section{Concluding remarks}\label{sec5}
In this paper, we investigate three first-order and two second-order fully discrete numerical schemes for the N-CH equation with analytic nonlinearity. The first-order schemes comprise the backward Euler scheme, the convex splitting scheme and the SSI1 scheme, while the second-order schemes consist of the BDF2 scheme and 2LI scheme. A second-order finite difference method is used for spatial discretization. The unique solvability, discrete mass conservation and (modified) energy stability are proved for our proposed numerical schemes. Furthermore, by means of energy stability and the {\L}ojasiewicz inequality, we provide rigorous proofs of converge to equilibrium as time goes to infinity for the sequences generate by our proposed numerical schemes.

\bibliographystyle{plainnat}
\bibliography{NCH1}
\end{document}